%% file: main_duality.tex
\documentclass[11pt, reqno]{amsart}

\usepackage[margin=1.3in]{geometry}

\usepackage{graphicx}
\usepackage{subcaption}
\usepackage{stmaryrd}
\usepackage{amssymb}
\usepackage{amsthm}
\usepackage{dsfont}
\usepackage{hyperref}
\usepackage{mathrsfs}
\usepackage{stmaryrd}
\usepackage[lite,initials,msc-links, bibtex-style]{amsrefs}
\usepackage{amsmath}
\usepackage{centernot}
\usepackage{enumerate}
\usepackage{xcolor}
\usepackage{verbatim}

\newtheorem{theorem}{Theorem}[section]
\newtheorem{lemma}[theorem]{Lemma}
\newtheorem{proposition}[theorem]{Proposition}
\newtheorem{corollary}[theorem]{Corollary}
\numberwithin{equation}{section}

\theoremstyle{definition}
\newtheorem{definition}{Definition}

\newtheorem*{notation}{Notation}
\newtheorem{example}{Example}
\theoremstyle{remark}
\newtheorem{remark}{Remark}

\newcommand{\ZZ}{\mathbb{Z}}
\newcommand{\PP}{\nu}
\newcommand{\E}{\nu}
\newcommand{\id}{\mathds{1}}
\newcommand{\NN}{\mathbb{N}}
\newcommand{\RR}{\mathbb{R}}

\renewcommand{\d}{\mathrm{d}}

\renewcommand{\c}[1]{\ensuremath{\mathcal{#1}}}

\newcommand{\Group}{{\ensuremath{\mathbb{G}}}}
\newcommand{\Green}{\ensuremath{\mathbf{G}}}
\newcommand{\J}{J}
\newcommand{\n}{\mathbf{n}}

\newcommand{\exact}{\star}
\newcommand{\coclosed}{\Diamond}
\newcommand{\Hcycle}{\ensuremath{H_{\star}}}
\newcommand{\Hstar}{\ensuremath{H_{\Diamond}} }

\title{On the duality between height functions and continuous spin models}

\author{Diederik van Engelenburg \and Marcin Lis}

\date{\today}

\begin{document}
\maketitle 

\begin{abstract}
We revisit the classical phenomenon of duality between random integer-valued height functions with positive definite potentials and abelian spin models with O(2) symmetry. We use it to derive 
new results in quite high generality including: a universal upper bound on the variance of the height function in terms of the Green's function (a GFF bound) which among others implies localisation on transient graphs;
monotonicity of said variance with respect to a natural temperature parameter; the fact that delocalisation of the height function implies a BKT phase transition in planar models; and
also delocalisation itself for height functions on periodic ``almost'' planar graphs. 
\end{abstract}

\input{sections/introduction}
\input{sections/general_dual}
\input{sections/upperbound}

\input{sections/bkt}
\input{sections/monotonicity}

\input{sections/delocalization}
\input{sections/remainder}
\newpage
\appendix
\input{sections/appendix}

\bibliography{PosDef}

\end{document}

%% file: sections/introduction.tex

\section{Introduction}
The phenomenon of duality in statistical mechanics goes back to the famous work of Kramers and Wannier who discovered an exact identity 
between the partition functions of the Ising model on a finite planar graph and an Ising model (at a different temperature) on its dual graph~\cite{KW}.
They used it to identify the {self-dual} temperature (that stays invariant under the duality transfomation) as the point of phase transition in the model on the square lattice (that is itself a self-dual graph). 
In an extended version of this correspondence, spin correlation functions are mapped to correlators of dual disorder variables introduced by Kadanoff and Ceva~\cite{KC}. 
This construction has been very fruitful in the study of the Ising model. A notable example are the works of Smirnov~\cite{smirnov}, and Chelkak and Smirnov~\cite{CheSmi}, who derived scaling limits of certain variants of order-disorder correlations (the so called fermionic observables), providing the first proofs of conformal invariance of the critical Ising model.

It is by now classical that analogous duality transformations exist for more general spin models 
whose state space is a locally compact abelian group \cite{WuWa, FroSpe, Savit} (we refer to~\cite{dubedat2011topics} for an introductory account). 
In such a setting Fourier transforms can be used to map one model with values in a group $\mathbb G$ to another model with values in the Pontryagin-dual group $\widehat {\mathbb G}$. For example, for $\mathbb G= {\mathbb Z/q \mathbb Z}$, $ q\in \mathbb N$,
duality was a crucial tool in the study of the planar $q$-state Potts model, the associated random cluster model, and the Ashkin--Teller model (see e.g.~\cite{BefDum,PfiVel,GlaPel,Lis19,Lis21,ADG}). These groups are self-dual in the sense that $\mathbb G\cong \widehat{\mathbb G}$, 
and moreover the models are dual to the same model on the dual graph, but with possibly different coupling constants. Another famous self-dual example is $\mathbb G=\mathbb R$ together with the discrete Gaussian free field, where duality exchanges electric and magnetic operators of the field (see e.g.~\cite{dubedat2011topics}).

In this article we go beyond the self-dual domain and consider the mutually dual but distinct groups of the integers ${\mathbb Z}$ and the circle~$\mathbb S$. 
This results in two very different objects facing each other on the opposite sides of the duality relation: one is a discrete \emph{random height function} with an unbounded set of values, and the other is a \emph{continuous spin model} with spins in the circle.
Duality can be then used to transfer probabilistic information from one side to the other.
 A landmark application of this relation appeared in the work of Fr\"ohlich and Spencer~\cite{FroSpe} who rigorously established the Berezinskii--Kosterlitz--Thouless (BKT) phase transition in the classical XY model on the square lattice (see Section~\ref{sec:bkt} for more background). They first showed delocalisation of the associated height function, and then used duality to conclude that spin correlations decay at most polynomially fast. New proofs of the latter implication appeared recently in~\cite{vEnLis,AHPS} which together with a novel approach to delocalisation introduced by Lammers~\cite{Lammers} yields alternative proofs of the BKT transition. All three proofs of~\cite{FroSpe,vEnLis,AHPS} use duality ``in the same direction'' in that they study spin correlations via disorder correlations in the dual height function. This leads to technical complications as disorders are nonocal functions of the height field.
 In~\cite{vEnLis,AHPS} these issues were taken care of by considering different graphical representations of the models. Here we argue that following duality ``in the opposite direction''
 leads to an even more concise proof (that only uses duality itself) of the implication that delocalisation of the height function implies the BKT transition in the spin model.  
Indeed, when the disorders appear on the spin model side, one can ``localise'' them by simply using the Taylor expansion to the second order, which is clearly impossible when disorders 
come as discrete excitations of the heights. 

Duality is an exact correspondence, and hence one expects that the critical point of the localisation-delocalisation phase transition is dual to the BKT critical point.
This was recently proved for the XY and the Villain model by Lammers~\cite{Lam23}. Here we also provide a result in the same direction for a larger class of models that includes the XY model. 
To be more precise, we establish an equivalence between the delocalisation of the height function and the divergence of a certain series (a type of susceptibility) of correlation functions in the spin model.

 Another contribution of this paper is a universal upper bound on the variance of the height function in terms of the variance of the discrete GFF. This holds true for all height function models with positive definite potentials, and moreover irrespectively of the graph being planar or not. 
There are two main applications. In the planar case, e.g., on $\mathbb Z^2$, this leads to a conjecturally optimal (up to a constant) logarithmic in the size of the system upper bound 
when the height function is delocalised. On the other hand, it shows that on transient graphs, e.g. on $\mathbb Z^d$, $d\geq 3$, the variance is always uniformly bounded and the height function is localised.

For a special class of potentials, we also establish monotonicity of the variance of the height function in a natural temperature parameter. 
As far as we know, the only available result of this type is for the integer-valued GFF \cite{KP}. 
We achieve {this} by transporting, through duality, the (appropriately generalised) Ginibre inequalities. 
One consequence is a direct proof of the fact that for the XY height function 
there is only one point of phase transition from a localised to a delocalised regime. 
A (more involved) proof of this fact was first given in~\cite{Lam23}.
Together with the dichotomy of Lammers \cite{Lammers22}, this also shows that for the XY height function on $\ZZ^2$, the transition is sharp.
Another application is that for two-dimensional Euclidean lattices with non nearest neighbour interactions, 
the height function undergoes a localisation-delocalisatoin phase transition.
  
We note that we only consider height functions with positive definite potentials, i.e., those that have well defined dual spin models, and vice versa.
The article is organised as follows:
 \begin{itemize}
 \item In Section~\ref{sec:gendual} we recall the notion of duality, and state in Lemma~\ref{L: covariance} its consequence for the covariances of the gradient of the height function and 
 and gradient of the spin model.
 This is the stepping stone to the remaining results in this article.
 \item In Section~\ref{sec:upperbound} we establish an upper bound on the variance of the height function in terms of the Green's function of the underlying simple random walk.
 The bound is valid on any, not necessarily planar graph, and in particular implies localisation of the height function on graphs on which simple random walk is transient.
 \item In Section~\ref{sec:bkt} in Theorem~\ref{thm:BKT} we give a direct proof of the fact that in two dimensions, delocalisation of the height function implies a BKT phase transition in the spin model in the sense that certain spin correlation functions are not summable. In Corollary~\ref{C: deloc BKT} using classical correlation inequalities we translate this to
 an analogous statement for the standard two-point functions, recovering the main result of~\cite{vEnLis}. Finally, for a subclass of spin models (that includes the classical XY model)
 we show that the above mentioned implication is actually an equivalence.
 \item In Section~\ref{sec:monot} in Theorem \ref{T: var_increasing_height} we show that for a certain class of height functions, the variance is increasing in the inverse temperature. 
 	We use this to prove that a phase transition occurs for these random height functions, when the underlying graph is planar or ``almost planar''. 
 \item In Section~\ref{sec:projection}, using only duality, we show that a certain (non-local) observable of a classical spin model has, up to multiplicative constants, the covariance of the discrete Gaussian free fiel. 
 	Remarkably, this holds for all graphs and does not depend on the temperature.
\item In Section~\ref{sec:CLT} we show a central limit theorem in the planar spin model that holds irrespectively of the temperature.
  \item In Appendix~\ref{ap:duality} we recall and give concise proofs of the main results needed for duality.
 \item In Appendix~\ref{ap:Gin} we extend the Ginibre inequalities to the setting that we need in Section~\ref{sec:monot}.
 \item In Appendix~\ref{ap:RP} we review the notion of reflection positivity that we use in Section~\ref{sec:bkt}.
  \item In Appendix~\ref{ap:div} we review the certain aspects of positive definite functions.
 \end{itemize}

 \subsection*{Acknowledgements} We thank Christophe Garban for useful remarks on an earlier version of the manuscript.
 The research of DvE was funded by the FWF (Austrian Science Fund) grant P33083: “Scaling limits in random conformal geometry”, and the research of ML
was supported by the FWF grant P36298-N: ``Spins, loops and fields''.
 

%% file: sections/general_dual.tex

\section{General duality}\label{sec:gendual}

\subsection{Discrete calculus}
We first give a basic background on discrete calculus on graphs, 
staying close to the language of \cite{LyPer}. 
Let $G = (V, E)$ be a locally finite graph and let $\Group$ be a group (we will consider $\Group= \RR, \ZZ$ with addition and $\Group= \mathbb{S}:=\{z\in \mathbb C: |z|=1\}$ with multiplication). 
To keep the exposition homogenous, we will use the additive notation for all considered groups.

A \emph{1-form} $\omega$ taking values in $\Group$ is an antisymmetric function defined on the directed edges $\vec E$ of $G$, 
i.e., such that $\omega_{vv'}= -\omega_{v'v}$, where $vv'$ denotes the directed edge $(v,v')$. The set of $1$-forms will be denoted by $\Omega^1(\Group) = \Omega^1( G,\Group)$, and the set $\mathbb G^V$ by $\Omega^0(\Group) = \Omega^0(G, \Group)$. 
We will identify the space of 1-forms with $\mathbb \Group^E$ by fixing, once and for all, one of the two orientations for each edge in~$E$. 
Define the \emph{boundary} operator $\d^*: \Omega^1(\Group) \to \Omega^0(\Group)$ by 
\begin{align*}
	\d^*\omega_x = \sum_{y\sim x} \omega_{yx},
\end{align*}
where $y\sim x$ indicates that $y$ and $x$ are adjacent in $G$, and the \emph{co-boundary} operator $\d: \Omega^0(\Group) \to \Omega^1(\Group)$ by
\[	
\d f_{xy} = f_y - f_x.
\]
Note that $\d^*$ and $\d$ are homomorphisms between groups $\mathbb G^E$ and $\mathbb G^V$, and hence we can define groups
\[
\Hcycle(\Group) = \Hcycle(G,\Group):= \mathrm{Im}(\d) \cong \Group^V / \ker(\d) \quad \text{ and } \quad \Hstar(\Group) =  \Hstar(G,\Group) :=\ker(\d^*).
\]
For $\Group=\mathbb S$, we will write $d \J$ to be the Haar probability measure on the induced (compact) groups $\Hcycle(\mathbb S)$ and $\Hstar(\mathbb S)$. 
If $\Group$ is only locally compact, the Haar measure is defined up to a multiplicative constant and we fix some normalisation. 
For a more concrete definition, we refer to Appendix \ref{ap:duality}. 
We make the convention that the space over which we integrate determines the measure.

\begin{notation}
In what follows we will use the letters $\epsilon, \omega$ (resp.\ $f,g,\tau$) to denote deterministic elements of $\Omega^1(\Group)$ (resp.\ $\Omega^0(\Group)$) when $\Group=\mathbb R$ or when $\mathbb G$ is not specified. 
We will write $\n$ and $h$ for (mostly random) elements of $\Omega^1(\mathbb Z)$ and $ \Omega^0(\mathbb Z)$ respectively, and $J$ and $\theta$ for (mostly random)
elements of $\Omega^1(\mathbb S)$ and $ \Omega^0(\mathbb S)$ respectively.

We will also often abuse notation in the following sense: through the identification $\exp( i\theta)\leftrightarrow\theta$, we have $\mathbb{S} \cong (-\pi, \pi]$, and
we will view $\J \in \Omega^1(\mathbb{S})$ as belonging to $\Omega^{1}(\RR)$. On the other hand, by considering real numbers modulo $2\pi$, 
we will map $\Omega^1(\RR)$ to $\Omega^1(\mathbb{S}^1)$.
One has to be careful when going from one space to the other: the embedding does not map $H_{\#} (\mathbb{S})$ to $H_{\#}(\RR)$, 
because for example a $1$-form $\omega \in \Omega^1(\mathbb{S})$ which satisfies $\d^* \omega= 0$ in $\mathbb{S}$, 
only satisfies $\d^*\omega = 0$ modulo $2\pi$ when viewed as a $1$-form in $\Omega^1(\mathbb{R})$. We will also think of $H_{\#}(\mathbb Z)$
as a subset of  $H_{\#}(\mathbb R)$ in the obvious way.
\end{notation}

\subsection{Spin models and random height functions}
In this section we consider a finite graph $G = (V,E)$. 
We will study random \emph{spin} and \emph{height} $1$-forms 
taking values in the spaces $H_{\#}(\mathbb S)$ and $H_{\#}(\mathbb Z)$ respectively for $\# \in \{\coclosed, \exact\}$. 

\begin{definition}[Height function and spin potentials] \label{def:potential}
Let $\mathcal V: \mathbb Z\to \mathbb R \cup \{ +\infty\}$ be symmetric, i.e., $\mathcal V (n)=\mathcal V(-n)$, such that 
\begin{align} \label{eq:Vsum}
\sum_{n\in \mathbb Z}n^2 \exp(-\mathcal V(n))<\infty,
\end{align}
and moreover such that $\exp(-\mathcal V)$
is \emph{positive definite}: for all $\alpha\in \mathbb R$, 
\begin{align}\label{eq:posdef}
	w(\alpha):=\exp(-\mathcal V(0))+\sum_{n=1}^{\infty}\exp(-\mathcal V(n))2\cos( n\alpha) > 0.
\end{align}
We call $\mathcal V$ the \emph{height function potential}, and $\mathcal U (\alpha) := -\log w(\alpha)$ the \emph{spin potential}.
\end{definition}
We will always assume that the considered potentials satisfy the conditions of Definition~\ref{def:potential}.
Note that condition~\eqref{eq:Vsum} implies that the series in \eqref{eq:posdef} is absolutely summable, and moreover that $\omega$, as well as $\mathcal U$, is twice continuously differentiable in $\alpha$.
In general, we will say that a function is positive definite if its Fourier transform is a nonnegative function.
This is not the classical definition of positive definiteness, but it is equivalent to it by Bochner's theorem.
\begin{example} \label{exple:potentials}The following potentials satisfy the conditions of Definition~\ref{def:potential}:
\begin{itemize}
\item  $\mathcal V(n)= -\log(I_{n}(\beta))$, where $I_n(\beta)$ is the modified Bessel function of the first kind, and $\c{U}(t) = -\beta \cos( t)$ for all $\beta>0$ is the potential of the \emph{classical XY model},
\item $\mathcal V(n)=\beta n^2$ for all $\beta >0$ is the potential of the \emph{integer-valued Gaussian free field} and the corresponding $\mathcal U$ defined through the series in~\eqref{eq:posdef} is the potential of the \emph{Villain spin model},
\item  $\mathcal V(n)=\beta \mathbf 1\{ n=\pm 1\}+\infty  \mathbf 1\{ |n|>1\}$ for $\exp(-\beta)<1/2$ is a model of random (nonuniform) \emph{Lipschitz functions}.
\item
	Any \emph{annealed Gaussian} potential $\c{V}$ meaning that there exists a finite Borel measure $\lambda$ on $[0, \infty)$ such that
\[
	e^{-\c{V}(n)} = \int_{[0, \infty)} e^{-\frac{\gamma}{2} n^2} \lambda(d \gamma)
\]
for all $n$. 
It satisfies Definition \ref{def:potential} because the function $n \mapsto \frac{\gamma}{2} n^2$ does and because by dominated convergence, we can exchange the integral and the summation in \eqref{eq:posdef}.  
This class includes the potentials $\c{V}(n) = \beta |n|^a$ for any $a \in (0, 2]$ (see~\cite{AHPS}). 
\end{itemize}
\end{example}

Let $\omega$ be as in Definition~\ref{def:potential}. Fix $\# \in \{\coclosed, \exact\}$, and consider a probability measure on \emph{spin $1$-forms} $J\in H_{\#}(\mathbb{S})$ defined by
\begin{align} \label{def:xy}
	d\mu_{\#}(\J)=d \mu_{G, \#}(\J) =\frac{1}{Z_{\#}} \Big(\prod_{e \in E}w(\J_e) \Big) d \J,
\end{align}
where $Z_{\#}$ is the \emph{partition function}, and $d \J$ denotes the Haar probability measure on the group $H_{\#}(\mathbb{S})$. 
For a 1-form $\epsilon \in \Omega^1(\RR)$, we define the \emph{twisted partition function}
\begin{align*}
	Z_{\#}(\epsilon) = \int_{H_{\#}(\mathbb{S})} \prod_{e\in  E} w  (\J_{e}+\epsilon_e ) d \J, 
\end{align*}
and note that $Z_{\#}(0)=Z_{\#}$.
We also define a probability measure on \emph{height $1$-forms} $\n \in H_{\#}(\ZZ)$ by
\begin{align} \label{eq:defh}
	\nu_{\#}(\n) =\nu_{G, \#} (\n) \propto \exp\Big(-\sum_{vv'\in E}\mathcal V(\n_{vv'})\Big).
\end{align}
Note that this is well defined as the normalisation constant is finite by assumption~\eqref{eq:Vsum}. 


For $f,g\in  \Omega^1(\mathbb R)$ and $\epsilon, \omega\in \Omega^1(\mathbb R)$, we will write 
\[
(f,g)_{\Omega^0}= \sum_{v\in V} f_vg_v, \qquad \text{and} \qquad (\epsilon, \omega)_{\Omega^1} = \frac{1}{2}\sum_{\vec{e} \in \vec{E}} \epsilon_{\vec{e}} \:\omega_{\vec{e}} =\sum_{{e} \in {E}} \epsilon_{{e}} \:\omega_{{e}} 
\]
for the standard inner products. We will usually drop the subscripts and simply write $(\cdot,\cdot)$ in case the space is clear from the context.

The central result that we will use is the following duality formula. Even though it is classical (see e.g.\ Appendix A in~\cite{FroSpe}), we will provide its derivation in Appendix~\ref{ap:duality}.
\begin{lemma}[Fourier--Pontryagin duality] \label{lem:duality}
	Let $\# \in \{\coclosed, \exact\}$ and let $-\#$ denote the other element of $ \{\coclosed, \exact\}$. 
	Then for any $\epsilon \in \Omega^1(\RR)$, we have
	\[
	\nu_{-\#} [\exp({ i (\n, \epsilon)}) ] = \frac{Z_{\#}(\epsilon)}{Z_{\#}}=\mu_{\#}\Big[\prod_{e\in E} \frac{w  (\J_e+\epsilon_e )}{w  (\J_{e})}\Big].
	\]
\end{lemma}

Clearly there are two intertwined random objects in the statement of Lemma~\ref{lem:duality}: the height and spin $1$-forms $\n$ and $J$ respectively. We will mostly apply the duality to analyse one of these two models
whose values are the \emph{exact} $1$-forms $H_\exact(\mathbb G)$, since then for each $\omega \in H_\exact(\mathbb G)$, there exists a unique $\tau \in \mathbb G^V$ such that
\[
\d\tau = \omega \qquad \text{ and } \qquad \tau_\partial =0,
\]
 where $\partial\in V$ is a fixed \emph{boundary} vertex of $G$, and $0$ is the identity element of $\mathbb G$. The random configuration $\tau$ is then distributed as a classical spin system with spins assigned to vertices with $0$ boundary conditions at $\partial$, and that interact through edges.

\begin{remark} 
In two dimensions there is a special form of duality where $\Hstar$ on the planar graph $G$ can be seen as $\Hcycle$ on the \emph{planar dual} graph $G^*$ by simply rotating all directed edges by $\pi/2$ to the left.
	Therefore if $\omega$ is a $1$-form such that $\d^*\omega = 0$, 
	there exists a function $\tau$ on the vertices of the {dual graph} $G^*$ (faces of $G$) which has $\omega$ as its gradient, i.e. 
	\begin{align*}
		\omega_{vv'} = \tau_u-\tau_{u'} = \d \tau_{uu'},
	\end{align*}
	where $u,u'$ are the two faces adjacent to $vv'$ from the right and left respectively. In this case, both models in Lemma~\ref{lem:duality} can be seen as classical spin and height function models.
\end{remark}

\begin{remark}
	As mentioned in the introduction, the Fourier--Pontryagin duality is usually applied in the opposite direction to Lemma~\ref{lem:duality}, i.e., to compute the characteristic function of the spin model rather than the height function. On the height function side this results in expectations of nonlocal observables (disorders) which are in general difficult to analyse.
	In our case however the disorder appears on the spin model side, and can be removed from the picture by taking derivatives at zero of the characteristic function.
	This is the main point of view which allows to obtain most of the results in this article using comparatively elementary arguments.
\end{remark}

One of the main tools in this article is the following identity. Even though it is a rather direct consequence of duality, we were unable to find this formulation in the literature.

\begin{lemma}[Covariance duality]\label{L: covariance} 
	Let $\# \in \{\coclosed, \exact\}$ and let $-\#$ be the other element of $ \{\coclosed, \exact\}$. For any $\epsilon, \omega \in \Omega^1(\RR)$, we have
	\[
 \E_{\#}\big[(\n, \epsilon)(\n, \omega)\big] + \mu_{-\#} \big[(\c{U}'(\J), \epsilon)(\c{U}'(\J), \omega)\big] = \sum_{e \in E} \mu_{-\#}[\c{U}''(\J_e)] \epsilon_e\omega_e.
	\]
\end{lemma}

\begin{proof}
It is enough to compute 
\[
\frac\partial {\partial s} \frac\partial {\partial t}	\Big(\nu_{\#} [\exp({i (h,s \epsilon +t \omega)}) ] \Big) \Big|_{s =t =0}
\]
by differentiating under the sign of integration on the right-hand side of the formula from Lemma~\ref{lem:duality}.
\end{proof}

\begin{remark}
In the case of measures $\nu_{\exact}$ on true height functions $h$, the quantity $\E_{\exact}[(\n, \epsilon)(\n, \omega)]=\E_{\exact}[(h,\d^* \epsilon)(h,\d^* \omega)]$ explicitly depends only on $\d^*\epsilon$ and $\d^*\omega$, and hence the rest of the equation above does so implicitly.
\end{remark}

Choosing $\epsilon = \id_{xy} - \id_{yx}$ for some edge $xy$ and $\tau = \id_{uv} - \id_{vu}$ for another edge $uv$, as a corollary we immediately get the following identities:
\begin{align} \label{eq:pontwiseD1}
\E_{\#}[\n_{xy}^2] = \mu_{-\#}[ \mathcal U''(J_{xy})- \mathcal U'(J_{xy})^2]
\end{align}
and
\begin{align} \label{eq:pontwiseD2}
\E_{\#}[\n_{xy}\n_{uv}] = -\mu_{-\#}[\c{U}'(\J_{xy})\c{U}'(\J_{uv})].
\end{align}

\begin{remark}
	Note that Lemma~\ref{L: covariance} implies that the sum of the covariance matrices of two mutually dual edge fields is diagonal, i.e., equals the covariance matrix of (possibly inhomogeneous) white noise. This was known for the discrete Gaussian Free Field ($\mathbb G=\mathbb R$), see e.g.~\cite{dubedat2011topics, Aru15} and Remark~\ref{rem:gff}, in which case the independent sum of the mutually dual edge fields is a collection of independent normal random variables. 
	The continuum analogue for the GFF can be found in \cite{AKM}. 
\end{remark}

\begin{remark} 
		Arguments similar in spirit to Lemma \ref{L: covariance}, but in the context of the Villain model appear in \cite{GaSe}. 
\end{remark}

Having established the covariance duality formula in Lemma~\ref{L: covariance}, we will now discuss several of its rather direct consequences. 
Unless stated otherwise, we study the models on a finite graph $G = (V,E)$ with a prescribed boundary vertex $\partial \in V$.
We will write $ \Omega_0^0(\mathbb G)$ for the set of functions $f\in \Omega^0(\mathbb G)$ with $f_\partial =0$.

%% file: sections/upperbound.tex

\section{Upper bound on the variance of the height function} \label{sec:upperbound}
In this section we consider random {exact} $1$-forms $\n \in \Hcycle(\mathbb Z) $ distributed according to~$\nu_\exact$. As mentioned before, for each such $1$-form $\n\in \Hcycle(\mathbb Z)$, there exists exactly one \emph{height function} $h\in \Omega_0^0(\mathbb Z)$ such that $ \d h =\n $.
Note that in the case of $\mathbb G=\mathbb R$, $\d$ and $\d^*$ are adjoint as linear operators, i.e., for all 
$f\in \Omega^0(\mathbb R)$ and $\omega\in\Omega^1(\mathbb R)$, we have
\begin{align} \label{eq:adjoint}
(f, \d^* \omega)_{\Omega^0} = (\d f,\omega)_{\Omega^1}.
\end{align}
Also note that the operator 
\[
\Delta:=\d^*\d: \Omega^0(\mathbb R) \to \Omega^0(\mathbb R)
\] 
is the \emph{graph Laplacian} on $G$, and it has a well defined inverse $\Delta^{-1}$ on $\Omega^0_0(\mathbb R)$. Moreover, as matrices,
\[
\Delta^{-1} =  \Green D^{-1},
\]
where $D=\textnormal{Diag}(\textnormal{deg}(v))_{v\in V\setminus \{ \partial\}}$ and $\Green$ is the Green's function of simple random walk on $G$ killed upon hitting $\partial$.

Let $f \in  \Omega^0_0(\mathbb R)$ and $\epsilon :=\d \Delta^{-1}f$ so that $\d^*\epsilon =f$. Discarding the explicitly nonnegative term $\mu_{\coclosed} [(\c{U}'(\J), \epsilon)^2]$ in Lemma~\ref{L: covariance} applied to $\epsilon=\omega$, and using \eqref{eq:adjoint} we get
\begin{align} \label{eq:upperbound}
	\nu_{ \exact}[(h, f)_{\Omega^0}^2] =\nu_{ \exact}[(\n, \epsilon)_{\Omega^1}^2] \leq \sum_{e \in  E} \epsilon_{e}^2 \big|\mu_{\coclosed}[\mathcal U''(\J_e) ]\big|  \leq C (\epsilon, \epsilon)_{\Omega^1},
\end{align}
where 
\[
C=\sup_{e\in E} |\mu_{\coclosed}[\mathcal U''(\J_e)]| \leq \sup_{J\in \mathbb S} |\mathcal U''(J)|<\infty.
\] 
On the other hand, by \eqref{eq:adjoint} again	$(\epsilon, \epsilon)_{\Omega^1}= (\d \Delta^{-1}f,\d \Delta^{-1}f)_{\Omega^1}=  ( \Delta^{-1}f,f)_{\Omega^0}$.

\begin{corollary}[GFF upper bound on variance]\label{cor:upper}
	For any $f\in \Omega^0_0(\mathbb R)$,
	\begin{align*}
		\nu_{\exact}[(h, f)^2] \leq  C  ( \Delta^{-1}f,f)_{\Omega^0}=C\sum_{v,v'\in V} f_vf_{v'}\frac{\Green(v, v')}{\deg(v')},
	\end{align*}
	where $C$ is as above.
\end{corollary}

\begin{remark} \label{rem:gff}
	One can also apply duality to the discrete Gaussian free field (GFF) (in this case both the primal and dual fields are real-valued as $\mathbb R$ is self-dual as a locally compact abelian group).
	The GFF is defined similarly to the integer-valued GFF with potential $\mathcal V(t)=t^2$ with the difference that the reference measure in~\eqref{eq:defh} is the Lebesgue measure on $\mathbb R$ and not the counting measure on $\mathbb Z$. 
	The model is self dual in that $\mathcal U(t)=\mathcal V(t)=t^2$, and in the analog of the corollary above actually get an equality since $( \epsilon, \mathcal U'(\J_e))=0$ since $\mathcal U'(\J_e) =2 \J_e \in H_{\coclosed}$, and $\d^* \epsilon \in H_{\exact}$ by definition. This agrees with the fact that the covariance of the GFF is given \emph{exactly} by the inverse Laplacian.
\end{remark}

\begin{remark} \label{rem:tightness}
	Consider an infinite countable graph $\Gamma=(\mathscr V, \mathscr E)$ and a sequence of increasing finite subgraphs exhausting $\Gamma$, i.e., $G_N \nearrow \Gamma$ as $N\to \infty$.
	If $f: \mathscr V\to \mathbb R$ has bounded support and mean zero, i.e., $\sum_{v\in \mathscr V} f_v=0$, where this sum is actually taken over a finite set of vertices, then we can find a 1-form $\epsilon$ on $\mathscr E$ with \emph{bounded} support such that $\d^* \epsilon=f$, and hence 
	\begin{align}\label{eq:local}
	\prod_{e\in E} \frac{w  (\J_{e}+\epsilon_e )}{w  (\J_{e})}=\prod_{e\in \textnormal{supp} (\epsilon)} \frac{w  (\J_{e}+\epsilon_e )}{w  (\J_{e})}
	\end{align}
	is a local bounded continuous function of $\J$ (in the product topology on $\mathbb S^{\mathscr E}$) whenever $\mathcal V$ and $w$ are as in Definition~\ref{def:potential}. Moreover, since $\mathbb S$ is compact metrizable so is $\mathbb S^{\mathscr E}$ with the product topology by Tychonoff's theorem, 
	and hence the edge spin models $\mu_{G_N,\#}$ always form a tight sequence of measures on $\mathbb S^{\mathscr E}$ as $N\to \infty$. 
	This in particular implies that $\mu_{G_N,\coclosed}$ converges weakly along a subsequence.
	Therefore Lemma~\ref{lem:duality} together with \eqref{eq:local} and the fact that
	\[
		\nu_{G_N,\exact} [\exp({i (\n, \epsilon)}) ]=\nu_{G_N,\exact} [\exp({ i (f, h)}) ]
	\]  
	for $N$ large enough so that $G_N$ contains $\textnormal{supp} (\epsilon)$, imply that the random height $1$-forms $\n$ under $\nu_{G_N,\exact}$, and hence also the \emph{differences} of the associated height function $h$, converge weakly along the same subsequence. 
	
One has to be careful as this is in general no longer true if $f$ does not have zero mean, e.g., $f=\delta_v$. Then $\epsilon$ with $\d^* \epsilon=f$ cannot be taken with bounded support 
(there always has to be an infinite path with nonzero values of $\epsilon$).
In this case tightness may fail when \emph{delocalisation} of the height function arises, i.e., $\nu_{G_N,\exact}[(h, f)^2]= \nu_{G_N,\exact}[h_v^2]\to \infty$ 
as $N\to \infty$ (e.g.\ if $\Gamma$ is planar, see Section~\ref{sec:bkt}).
\end{remark}

We also immediately deduce that delocalisation of the height function does not happen on transient graphs for potentials as in Definition~\ref{def:potential}.
We note that our result, despite its simple proof, seems new in this generality, and that such behaviour is expected for a larger class of potentials. 
We also note that the special case of the integer-valued GFF follows from a stronger estimate proved by Fr\"{o}hlich and Park~\cite{FroPar} (see also~\cite{KP}).
Some results in this direction related to the intever-valued GFF can also be found in \cite{AHPS}.

To state the result, we briefly recall the notion of Gibbs measures and {gradient} Gibbs measures (we do it for height functions only, and the definition for spin models used later in the article is completely analogous).
From now on we assume that $\Gamma = (\mathscr V, \mathscr E)$ is a locally finite, infinite graph. 
For a finite set $\Lambda \subset \mathscr V$ write $E(\Lambda)$ for the set of edges with at least one vertex in $\Lambda$.  
Let $\varphi: \Lambda^c \to \ZZ$ be a function and define the probability measure $\mu_{\Lambda}^\varphi$ 
supported on $h: \mathscr{V} \to \ZZ$ satisfying $h \mid_{\Lambda^c} = \varphi$ by
\[
	\nu_{\Lambda}^{\varphi}(h) \propto \exp \Big(-\sum_{e \in E(\Lambda)} \c{V}(\d h_e)\Big). 
\]
In other words, $\nu_{\Lambda}^\varphi$ is the measure $\PP_{\exact}$ from \eqref{eq:defh} with $\varphi$-boundary conditions outside $\Lambda$. 
A probability measure $\nu$ supported on height functions $h: \mathscr V \to \ZZ$ is called a \textit{Gibbs measure} (on $\Gamma$ with respect to the potential $\c{V}$) if it satisfies the Dobrushin--Lanford--Ruelle (DLR) relations:
for all finite sets $\Lambda \subset \mathscr{V}$, 
\[
	\nu_{\Lambda}(\cdot) = \int_{\mathbb{Z}^{\mathscr V}} \nu_{\Lambda}^{\varphi}(\cdot) d \nu(\varphi), 
\]
where $\nu_{\Lambda}$ denotes the restriction of $\nu$ to $\Lambda$.
If $\Gamma$ is a Cayley graph and the measure $\nu$ is invariant under shifts, it is called translation invariant. 
In terms of Gibbs measures, delocalisation corresponds to non-existence of translation invariant Gibbs measures. 

A \emph{gradient} Gibbs measure is a slight variation of the above, where we consider measures supported only on gradients. 
Fix this time a finite set of edges $\Lambda \subset \mathscr{E}$. Let $\omega$ be an exact $1$-form (thus taking value in $\Hcycle(\mathscr{E}, \ZZ)$). 
Define the probability measure $\mu_{\Lambda}^{\omega}$ supported on $1$-forms $\n \in \Hcycle(\mathscr{E}, \ZZ)$ satisfying $h \mid_{\Lambda^c} = \omega \mid_{\Lambda^c}$ as
\[
	\mu_{\Lambda}^{\omega}(\n) \propto \exp\Big(-\sum_{e \in \Lambda} \c{V}(\n_e)\Big). 
\]
A probability measure supported on $1$-forms $\n \in \Hcycle(\mathscr{E}, \ZZ)$ will be called a gradient Gibbs measure if it satisfies the analog of the DLR equation above
in this setup. 

\begin{theorem} \label{thm:extransinv}
	Let $\Gamma = (\mathscr V,\mathscr E)$ be a transient graph and $\c{V}$ a height function potential as in Definition~\ref{def:potential}. 
	Then, there exists an infinite volume Gibbs measure on $\Gamma$ with respect to $\mathcal V$. If $\Gamma$ is moreover an amenable Cayley graph, 
	there exist translation invariant Gibbs measures. 
\end{theorem}

\begin{proof}
Let now $G_N\nearrow \Gamma$, as $N\to \infty$ be an exhaustion of $\Gamma$ by finite subgraphs $G_N$. 
Define the boundary $\partial_N := \partial G_N$ to be the set of vertices in $ G_N$ adjacent to a vertex from outside of $ G_N$.
Let $\E_{G_N, \exact}[\cdot]$ be the expectation associated with the height function on $ G_N$ with $0$-boundary conditions. 
Fix any $\Lambda \subset \c{V}$ finite and let $f: \Lambda \to \RR$ be any function.
It follows from Corollary~\ref{cor:upper} that 
\[
	\E_{G_N,\exact}[(h, f)^2] \leq C \max_{v \in \Lambda} \frac{\Green_{N} (v,v)}{\deg(v)} (f, f)_{\Lambda}. 
\]
where $C < \infty$, and $\Green_N$ is the Green's function of simple random walk on $G_N$ killed on hitting $\partial_N$. 
Since $\Gamma$ is transient, the right-hand side is uniformly bounded in $N$. 
Therefore, the sequence $\PP_{G_N, \exact}(h |_{\Lambda})$ is tight and subsequential limits exist by Prokhorov's theorem. 
By a diagonal argument, we can extract a further sub-sequence $N_K$ so that $\PP_{G_{N_K}, \exact}(h|_{\Lambda})$ converges for each
$\Lambda$ finite. 
Any such subsequential limit is a Gibbs measure as it satisfies the DLR relations. 
This proves the first part of the theorem. 

For the second part, suppose that $\Gamma$ is an amenable Cayley graph, 
so that 
\(
	\E_{G_N, \exact}[h_v^2] \leq C',
\)
for some $C' < \infty$ which is independent of $v$ and the exhaustion $(G_N)_{N\geq 1}$. 
Let $\mu$ be a subsequential limit (which exists by the argument above, and is a Gibbs measure). 
Let $o \in \mathscr V$. Since $\Gamma$ is amenable, there is some F\o lner sequence (also called Van Hove sequence) $(F_N)_{N \geq 1}$ {of sets of vertices} containing $o$. 
This means $F_N \nearrow \mathscr V$ and $|\partial F_N| / |F_N| \to 0$ as $N \to \infty$. 
Let
\[
	\nu_N := \frac{1}{|F_N|} \sum_{x \in F_N} \mu \circ \theta_x, 
\]
where $\theta_x$ is the shift towards $x$ (since $\Gamma$ is a Cayley graph, this is the same as left multiplication in the group).
This is a Gibbs measure because the set of Gibbs measures is closed under translations and convex combinations. 
Moreover, $\nu_N[h_v^2] \leq C'$ for each $N$ and $v \in \mathscr V$. Therefore, $(\nu_N)_{N\geq 1}$ is tight. 
Let $\nu$ be any subsequential limit, which is again a Gibbs measure. 
By construction and since $|\partial F_N| / |F_N| \to 0$, we have $\nu \circ \theta_x = \nu$ for each $x$, and hence $\nu$ is translation invariant.
\end{proof}

%% file: sections/bkt.tex

\section{Delocalisation implies the BKT phase transition in two dimensions} \label{sec:bkt}
\subsection{Background}
In this section we consider the spin and height function models on the square lattice $\mathbb Z^2$ and we show that delocalisation of the height function (defined below) is equivalent to the divergence of a certain series of two-point functions in the dual spin model. One of the conclusions is that delocalisation implies the Berezinskii--Kosterlitz--Thouless (BKT) phase transition in the dual spin model~\cite{Ber1,KosTho}. 

This implication for the classical XY and the Villain spin models, together with a proof of delocalisation of the associated height functions, was first obtained by Fr\"{o}hlich and Spencer in their seminal work establishing the BKT transition~\cite{FroSpe} (also see~\cite{KP} for an exhaustive account).
Recently alternative proofs were provided by Aizenman et al.~\cite{AHPS} (first for the Villain and later also for the XY model) and by the authors~\cite{vEnLis} for the XY model. Together with the new conceptual approach to delocalisation introduced by Lammers~\cite{Lammers}, these works improve our mathematical understanding of the BKT transition.
These results can be thought of as an inequality between the critical points of the mutually dual spin and height function models.
A natural conjecture is that these critical points always coincide. In the case of the XY and Villain model this was confirmed in a recent work Lammers~\cite{Lam23}.

In this section we provide yet another, and arguably the simplest so far, proof of the fact that delocalisation of the height function implies that correlations functions of certain observables in the spin model do not decay exponentially fast in the distance. 
For reflection positive models (which is the case when $-\mathcal U$ is itself positive definite, i.e., has nonnegative coefficients in the Fourier series), we moreover obtain an equivalence between delocalisation and nonsummability of spin correlations.
Our approach, unlike the previous ones, is based solely on duality, and does not invoke any additional (e.g.~graphical) representations of the models at hand.

\subsection{Notions of delocalisation}
It is now well established that integer-valued height functions on $\mathbb Z^2$ (or in general on periodic planar lattices) undergo a phase transition between a
\emph{localised} \emph{(smooth)} and a \emph{delocalised} \emph{(rough)} regime~\cite{FroSpe,CPST,DHLR,Lis21,DKMO,Lammers,LamOtt,Lammers22,Lam23}. 
We say that a potential $\mathcal V$ is localised (on $\mathbb Z^2$) if it admits a translation-invariant Gibbs measure on height functions $h: \mathbb Z^2\to \mathbb Z$. Otherwise we say that $\mathcal V$ delocalises.
It is known that if $\mathcal V$ is {convex} on the integers, and moreover its second discrete derivative is nonincreasing, i.e., $\mathcal V$ is a so-called \emph{supergaussian} potential~\cite{LamOtt,Lammers22}, then delocalisation in this sense is equivalent to the fact that 
\begin{align} \label{eq:vardeloc}
\sup_{N \geq 1} \nu_{\Lambda_N,\exact}[h_{\mathbf 0}^2] =\infty,
\end{align}
where $\mathbf 0$ is the origin of $\mathbb Z^2$, and $\Lambda_N=[-N,N]^2\cap \mathbb Z^2 $ where we identify all vertices in $\Lambda_N$ that are adjacent to $\Lambda^c_N:=\mathbb Z^2 \setminus \Lambda_N$ as one boundary vertex $\partial$ (wired boundary) and set $h_\partial =0$. Moreover for such potentials, the sequence in~\eqref{eq:vardeloc} is nondecreasing in~$N$~\cite{LamOtt}, and it grows up to a mulitplicative constant at least like $\log N$~\cite{Lammers22} (which is consistent with the general conjecture stating that delocalised height functions should behave like the GFF at large scales). 

Yet another approach to delocalisation is to work with infinite volume \emph{gradient} measures and study the variance of the increment of the height between two distant points. This was e.g.\ studied in~\cite{Lis19,Lis21} in the context of the six-vertex model and it will be convenient for us to follow the same route here, as we already know by Remark~\ref{rem:tightness} that translation invariant gradient Gibbs measures always exist for potentials as in Definition~\ref{def:potential}. 
We say that a potential $\mathcal V$ is $\nabla$-\emph{delocalised} (on $\mathbb Z^2$) if for any translation-invariant gradient Gibbs measure $\nu$ (with expectation $\nu$), we have 
\begin{align} \label{eq:graddeloc}
\sup_{v\in \mathbb Z^2} \nu[(h_v-h_{\mathbf 0})^2]= \infty.
\end{align}

\begin{lemma} \label{lem:nabla}
If a potential $\mathcal V$ is delocalised, then it is also $\nabla$-delocalised.
\end{lemma}
\begin{proof}
Suppose otherwise that there exists a translation-invariant gradient Gibbs measure~$\nu$ for which the supremum in~\eqref{eq:graddeloc} is finite.
Then, as in Theorem~\ref{thm:extransinv}, by considering convex combinations of translations of $\nu$ thought of as a measure on height functions $\tilde h$ given by $\tilde h_v=h_v-h_{\mathbf 0}$ we can construct a translation invariant Gibbs measure on height functions which is a contradiction. We leave the details to the reader.
\end{proof}

We note that the opposite implication is also true e.g.\ for potentials $\mathcal V$ that are convex on the integers. Indeed, in this case it is known from the foundational work of Sheffield~\cite{sheffield} that each Gibbs measure for height functions has a finite second moment (since the height at every point has a log-concave distribution).

\subsection{Setup}\label{sec:setup}
Let us fix mutually dual potentials $\mathcal V$ and $\mathcal U$ as in Definition~\ref{def:potential}.
It will be convenient to consider the spin and height function models on finite, exponentially growing tori $\mathbb T_N= (\mathbb Z / 2^N\mathbb Z)^2$. This way we achieve three properties by construction:
\begin{itemize}
\item we work with measures that are translation invariant and invariant under $\pi/2$-rotations,
\item we can apply the duality of Lemma~\ref{lem:duality} first in the finite volume $\mathbb T_N$, and then take simultaneous (subsequential) infinite-volume limits, $\mathbb T_N\to \mathbb Z^2$ as $N\to \infty$, on both sides of the duality relation,
\item we get an explicit monotonicity in $N$ for the Green's function of the random walk on $\mathbb T_N$ (see below).
\end{itemize}

Let $\mu=\mu_{\mathbb Z^2,\coclosed}$ be any subsequential limit of $\mu_{\mathbb T_{N}, \coclosed}$, and let $\nu =\nu_{\mathbb Z^2,\exact}$ denote the limit of $\nu_{\mathbb T_{N}, \exact}$ taken along the same subsequence (it exists by Remark~\ref{rem:tightness}).
One can think of $\nu$ as a probability measure on height functions $h:\mathbb Z^2\to \mathbb Z$ satisfying $h(\mathbf 0)=0$. By weak convergence, the duality of Lemma~\ref{lem:duality} holds also for $\mu$ and $\nu$ whenever $\epsilon \in \Omega^1(\mathbb Z^2,\RR)$ is of bounded support.
The same is true for Corollary~\ref{L: covariance} and Corollary~\ref{cor:upper}, where we choose $\partial = \mathbf 0$ and consider the Green's function of a random walk on $\mathbb Z^2$ killed at $\mathbf 0$.

Note that by planar duality, we have $H_\coclosed(\mathbb Z^2, \mathbb S) \cong H_\exact((\mathbb Z^2)^*, \mathbb S) $,
where $(\mathbb Z^2)^* \cong \mathbf 0^*+\mathbb Z^2$ with $\mathbf 0^*:=(1/2,1/2)$, is the \emph{dual} square lattice. 
Since $ H_\exact((\mathbb Z^2)^*, \mathbb S) \cong {\mathbb S}^{(\mathbb Z^2)^*\setminus \{ \mathbf 0^* \}}$, we can think of $\mu$ as a Gibbs measure on spin configurations $\theta$ on $(\mathbb Z^2)^*$ where the spin at $\mathbf 0^*$ is fixed to be the identity element of $\mathbb S$. 

Finally, let $v_n=(n,0)\in \mathbb Z^2$, and
let $p_n=(e_0,e_1,\ldots,e_{n-1})$ be the directed horizontal path from $v_0$ to $v_{n}$. We identify $p_n$ with the 1-form that assigns $1$ to each directed edge in $p_n$, 
and $0$ to the directed edges of $\mathbb Z^2$ that are not in $p_n$. For compactness of notation, we write $\J_i=\J_{e_i}$ and $h_i=h_{u_i}$.

\subsection{The implication}

Applying Lemma~\ref{L: covariance} and Corollary~\ref{cor:upper} in finite volume, and then taking the subsequential limit as in Section~\ref{sec:setup}, we have 
\begin{align} \label{eq:vanish}
	0\leq \sum_{i=0}^{n-1}   \mu[\mathcal U''(\J_i)] - \mu \Big[\Big(\sum_{i=0}^{n-1}\mathcal U'(\J_i)\Big)^2\Big]\leq \limsup_{N\to \infty}  \nu_{\mathbb T_{N}, \exact}[(h_{0} -h_{n})^2]\leq \limsup_{N\to \infty} \mathbf G_{N}(v_n,v_n),
\end{align}
where $\mathbf G_{N}$ is the Green's function of simple random walk on $\mathbb T_N$ killed at $\mathbf 0$. This is equivalent to a random walk on $\mathbb Z^2$ killed at all points in $2^N\mathbb Z^2$. Hence, $\mathbf G_{N}\nearrow \mathbf G$ as $N\to \infty$, where $\mathbf G$ is the Green's function of a random walk on $\mathbb Z^2$ killed at $\mathbf 0$.
Classically we have $\mathbf G(v_n,v_n)\leq \textnormal{const}\times\log n$ (see e.g.~\cite{LyPer}). Plugging this bound into \eqref{eq:vanish}, dividing both sides by $n$,
letting $n\to \infty$, and finally using translation invariance of $\mu$, we get
\begin{align}\label{eq:cesaro}
	\lim_{n\to \infty}\frac{1}{n}\sum_{k=1}^{n-1} u_k=\frac12 \mu[\mathcal U''(\J_0)-\mathcal U'(\J_0)^2] , \quad  \textnormal{where} \quad u_k=\sum_{i=1}^k \mu [\mathcal U'(\J_0)\mathcal U'(\J_i)].
\end{align}
In particular $u_k$ converges in the Ces\`aro sense as $k\to \infty$. 

\begin{theorem}[Delocalisation implies the BKT phase transition] \label{thm:BKT}
Consider the setup from Section~\ref{sec:setup}.
	If the height function delocalises in the sense that~\eqref{eq:graddeloc} holds true for~$\nu$, then
	\begin{align*}
		\sum_{i=1}^\infty  i |\mu [\mathcal U'(\J_0)\mathcal U'(\J_i)]|= \infty. 
	\end{align*}
	In particular, there is no exponential decay of the two-point function $\mu [\mathcal U'(\J_0)\mathcal U'(\J_i)]$ as $i\to \infty$.
\end{theorem}
\begin{proof}
	We can assume that $\sum_{i=1}^\infty |\mu [\mathcal U'(\J_0)\mathcal U'(\J_i)]|<\infty$ since otherwise we are done. This means that $u_k$ converges in the classical sense to its Ces\`aro limit
	from \eqref{eq:cesaro}. Hence, 
	\begin{align} \label{eq:exid}
		\mu[\mathcal U''(\J_0)-\mathcal U'(\J_0)^2] = \lim_{k\to \infty}2 u_k = 2\sum_{i=1}^\infty  \mu [\mathcal U'(\J_0)\mathcal U'(\J_i)].
	\end{align}
	By Lemma~\ref{L: covariance} applied in the infinite volume ($p_n$ has bounded support) and translation invariance of $\mu$, we have
	\begin{align*}
		\nu[(h_{n}-h_{ 0})^2] & =\sum_{i=0}^{n-1}\mu[\mathcal U''(\J_i)-\mathcal U'(\J_i)^2] -2\sum_{i=1}^{n-1} (n-i) \mu [\mathcal U'(\J_0)\mathcal U'(\J_i)] \\
		&= 2n \sum_{i=1}^\infty\mu [\mathcal U'(\J_0)\mathcal U'(\J_i)] -2\sum_{i=1}^{n-1} (n-i) \mu [\mathcal U'(\J_0)\mathcal U'(\J_i)]\\
		&= 2\sum_{i=1}^{n-1}i\mu [\mathcal U'(\J_0)\mathcal U'(\J_i)] +2n\sum_{i=n}^{\infty}  \mu [\mathcal U'(\J_0)\mathcal U'(\J_i)] \\
		&\leq 2 \sum_{i=1}^\infty  i |\mu [\mathcal U'(\J_0)\mathcal U'(\J_i)]|.
	\end{align*}
	By the assumption, and translation and $\pi/2$-rotation invariance of $\mathcal \nu$, we have
	\[
\infty =\sup_{v\in \mathbb Z^2} \nu[(h_{v}-h_{\mathbf 0})^2]\leq 2 \sup_{n\geq 1} \nu[(h_{n}-h_{\mathbf 0})^2],
	\]
which together with the inequality above finishes the proof.
\end{proof}

It is classical that spin correlation functions decay exponentially fast at high temperatures (here the temperature is incorporated in the definition of $\mathcal U$).
This in particular implies that $\sum_{i=1}^\infty i |\mu [\mathcal U'(\J_0)\mathcal U'(\J_i)]|<\infty$. 
From this point of view Theorem~\ref{thm:BKT} says that if the height function delocalises, then the associated
spin model undergoes a BKT phase transition from a regime with exponential decay to a regime with slow decay of correlations.

\subsection{The case of the $XY$ model}
The change of behaviour of the two-point functions $\mu[\mathcal U'(J_0) \mathcal U'(J_i)]$ as $i\to \infty$ clearly indicates a phase transition in the spin model. However it is more common to look at correlations of the type $\mu[\mathcal F(\theta_u-\theta_{u'})]$ when $u$ and $u'$ are far apart, where~$\theta$ is the underlying spin field on $(\mathbb Z^2)^*$ (the faces of $\mathbb Z^2$), and where $\mathcal F$ is some chosen function, e.g. $\mathcal F=\mathcal U$. 

For general spin models, it is not clear how to compare these two types of correlation functions. Here we present an approach based on correlation inequalities in the case of the classical XY model, i.e., when $\mathcal U(t)=- \beta\cos(t)$, where $\beta>0$ is the inverse temperature in the spin model.

To this end, consider the setup as in Theorem~\ref{thm:BKT}. If $\{u_i,u_i'\}$ is the dual edge of $e_i$, writing $\theta_i=\theta_{u_i}$ and $\theta_i'=\theta_{u_i}'$, we have
\begin{align} \label{eq:trigon}
	\frac2{\beta^2} \mu[\mathcal U'(e_0) \mathcal U'(e_i)] &= 2\mu[\sin(\theta_0-\theta_0') \sin(\theta_i-\theta_i')] 
	\\ &=\mu[\cos(\theta_0-\theta_0'-\theta_i+\theta_i')]-\mu[\cos(\theta_0-\theta_0'+\theta_i-\theta_i')] \nonumber
\end{align}
A version of the classical Ginibre inequality for the XY model~\cite{Gin} (see also~\cite{BLU}) says that
\begin{align*}
	\mu[ \sin ( \theta _0)\cos(\theta'_0) \sin(\theta_i) \cos(\theta'_i) ] \leq \mu[ \sin ( \theta _0) \sin(\theta_i)] \mu[ \cos(\theta'_0)\cos(\theta'_i) ],
\end{align*}
which after expanding into cosines of sums of angles and disregarding terms that are not invariant under global rotation (shift of angles mod $2\pi$) whose expectations vanish,
we obtain
\begin{align*}
	\mu[\cos(\theta_0+\theta_0'-\theta_i-\theta'_i)]& + \mu[\cos(\theta_0-\theta'_0-\theta_i+\theta'_i)]  -\mu[\cos(\theta_0-\theta'_0+\theta_i-\theta'_i)]  \\
	&\leq 2 \mu[ \cos ( \theta _0-\theta_i)] \mu[ \cos(\theta'_0-\theta'_i) ]  \\
	&= 2 \mu[ \cos ( \theta _0-\theta_i)]^2,
\end{align*}
where the last identity follows by reflection invariance of $\mu$ across the real line.
Analogous inequality follows by exchanging the roles of $\theta_i$ and $\theta'_i$, which results in swapping the signs of the second and third term in the first line.
Using that the first term is positive by the first Griffiths inequality, and combining with \eqref{eq:trigon}, we get that
\begin{align} \label{eq: inequality spin-spin}
	\tfrac12\mu[\cos(\theta_0+\theta_0'-\theta_i-\theta'_i)]+|\mu[\mathcal U'(e_0) \mathcal U'(e_i)] | & \leq \max\{ \mu[ \cos ( \theta _0-\theta_i)]^2,\mu[ \cos ( \theta _0-\theta'_i)]^2\} \nonumber
	\\ &= \mu[ \cos ( \theta _0-\theta_i)]^2,
\end{align}
where the last identity follows from the Messager--Miracle-Sole inequality~\cite{MMS} by applying the reflection across the real line.

These considerations, together with Lemma~\ref{lem:nabla}, lead us to the following corollary that recovers the main result of~\cite{vEnLis}. 

\begin{corollary}\label{C: deloc BKT}
	If the height function associated with the XY model on $\mathbb Z^2$ delocalises, then  
	\[
	\sum_{i=1}^\infty  i  \mu[ \cos ( \theta _0-\theta_i)]^2 \geq \sum_{i=1}^\infty  i  (\mu[ \cos ( \theta _0-\theta_i)]^2-\tfrac12\mu[\cos(\theta_0+\theta_0'-\theta_i-\theta'_i)]) = \infty.
	\] 
\end{corollary}

\begin{remark}
For the XY model it is known that there exists only one translation-invariant Gibbs measure $\mu$ on $\mathbb Z^2$~\cite{Pfister}, and hence regular, instead of subsequential, limits may be taken Section~\ref{sec:setup}.
\end{remark}

\subsection{An equivalence}
When $-\mathcal U$ is itself positive definite, i.e., all its Fourier coefficients are nonnegative, we can actually conclude more than in the above discussion. Indeed, in this case $\mu$ is reflection positive (see Appendix \ref{ap:RP}). This implies that 
\[
\mu [\mathcal U'(\J_0)\mathcal U'(\J_i)] \geq 0
\] 
as this holds true on $\mathbb T_N$ for every $N\geq i$ by reflection positivity. Therefore the Ces\`aro convergence from \eqref{eq:cesaro} implies classical convergence, and \eqref{eq:exid} always holds true.
The same argument as in the proof of Theorem~\ref{thm:BKT} yields the following corollary.
\begin{corollary} \label{cor:equiv}
Consider the setup from Section~\ref{sec:setup}, and moreover assume that $-\mathcal U$ is positive definite. 
Then~\eqref{eq:graddeloc} holds true for~$\nu$ if and only if
	\begin{align} \label{eq:infinite}
		\sum_{i=1}^\infty  i \mu [\mathcal U'(\J_0)\mathcal U'(\J_i)]= \infty. 
	\end{align}
\end{corollary}

\begin{remark}
The identity from \eqref{eq:exid} can be rewritten in a more symmetric form as
\begin{align} \label{eq:finite}
\sum_{i\in \mathbb Z}  \mu [\mathcal U'(\J_0)\mathcal U'(\J_i)]=\mu[\mathcal U''(\J_0)],
\end{align}
where now the sum is over a bi-infinite path of edges. Curiously, this is an exact (but nonlocal) identity for correlation functions that is independent of the (hidden in the definition of $\mathcal U$) 
temperature parameter. In particular the series in \eqref{eq:finite} is always convergent, independently of the temperature. This is in contrast with the behaviour of the series in~\eqref{eq:infinite}
that does undergo a phase transition. This, together with the relation to the gradient of the heigh function~\eqref{eq:pontwiseD2}, is consistent with the conjecture that in delocalised 
regime the discrete GFF describes the large-scale fluctuations of the height function. Indeed, the two-point function of the gradient of the discrete GFF is known to decay like the inverse square of the distance (see e.g.~\cite{BauWeb}).
\end{remark}

%% file: sections/monotonicity.tex

\section{Monotonicity of variance of the height function in temperature} \label{sec:monot}
In this section we will show that the variance of the height function is monotone in a natural temperature parameter under some further assumptions on the potential.
To the best of our knowledge the result is new, even in the case of planar graphs.
Together with the dichotomy of Lammers \cite{Lammers22}, this directly implies that the height function of the XY model undergoes a sharp phase transition on the square lattice. 

\subsection{Setup} \label{sec:setup2} 
Let $G = (V, E)$ be a finite graph. 
To each edge $e \in E$, associate
\begin{itemize}
	\item a twice continuously differentiable spin potential $\c{U}_e: \mathbb{S} \to \RR$ such that $-\mathcal U_e$ is positive definite, 
	\item a non-negative real $\beta_e$ (thought of as the inverse temperature in the spin model), 
	\item the dual potential $\c{V}_{\beta_e} := \c{V}_{\beta_e, e}$ of $\beta_e \c{U}_e$ as in Definition \ref{def:potential}. 
\end{itemize}
We will consider the family of measures $\nu_{\beta, \exact}$ 
for height functions and their dual measures $\mu_{\beta, \coclosed}$, indexed by $\beta = (\beta_e)_{e \in E}$. 

We wish to point out that the above requirements on the spin potential $\c{U}$ can also be described purely in terms of the height function potential:
if $\c{V}$ satisfies the conditions of Definition \ref{def:potential} and $e^{-\c{V}}$ is infinitely divisible (in the sense that each division satisfies Definition \ref{def:potential}), then the corresponding spin potential satisfies the above conditions, see Appendix \ref{ap:div}. 
This equivalence is not important in the remainder of this section. 

\begin{remark}
		The setup here is less general than in the rest of this text, 
		as we need to make a further assumption on the potentials $\c{U}$ (or equivalently on the potentials $\c{V}$ as explained in Appendix \ref{ap:div}). 
\end{remark}

\subsection{Increasing variance}
The main result of this section is the following fact. 
\begin{theorem} \label{T: var_increasing_height}
	Consider the setup as in Section \ref{sec:setup2}. For each $x, y \in V$ and $e \in E$ the function
	\[
		\beta_e \mapsto \nu_{\beta, \exact}[(h_x - h_y)^2]
	\]
	is non-decreasing. 
\end{theorem}

\begin{example} Let us first give some examples of potentials to which this theorem applies. 
	\begin{itemize}
		\item In case of the classical XY model, $-\c{U}(t) = \cos(2\pi t)$ which is positive definite. 
		\item A generalisation of this is any height function that is dual to a spin system with a positive definite potential.
		\item The Gaussian potential $k \mapsto \frac{1}{2\beta} k^2$ does itself not fall into this class. 
		However, it arises as a (rescaled) limit of the XY height function potentials so the conclusion of the above theorem still holds. 
		We will show the result on the height-function side, but it was already known on the dual spin side, see e.g. \cite{NewWu, AHPS}. 
		Let $G$ be a finite graph and let $e \in E$ be some fixed edge. 
		Replace $e$ by $N$ copies of $e$ and set on each copy of this edge the XY potential $\c{V}_{\beta N}(k) = -\log(I_{k}(\beta N))$. 
		Note that the effective height-function potential on the edge $e$ is equal to
		\[
			e^{-\c{V}^{\mathrm{eff}}_N(k)} = I_k(\beta N)^N. 
		\]
		On the other hand, the modified Bessel function $I_k(x)$ has the expansion as $N\to \infty$:
		\[
			I_k(\beta N) = \frac{e^{\beta N}}{ \sqrt{2\pi \beta N}} \left(1 - \frac{4k^2 - 1}{8\beta N} + O(1 / N^2)\right),
		\]
		which goes back to \cite{Kirch_bes}. In particular, as $N \to \infty$, 
		\[
			I_k(\beta N)^N \sim \left(\frac{e^{\beta N}}{ \sqrt{2\pi \beta N}}\right)^N e^{\frac{1}{8 \beta}} e^{- \frac{1}{2\beta}k^2} =: c_{\beta, N} e^{-\frac{1}{2\beta} k^2}, 
		\]
		where the constant $c_{\beta, N}$ does not depend on $k$. 
		Note that for Gibbs measures, the constant $c_{\beta, N}$ corresponds to a global shift of the potential $\c{V}^{\mathrm{eff}}_N$ which does not change the model. 
		Hence, since for each $N$ we can apply Theorem \ref{T: var_increasing_height}, the result holds in the limit as $N \to \infty$, where the potential on the edge $e$ equals $k \mapsto \frac{1}{2\beta}k^2$ as desired. 
	\end{itemize}
\end{example}

To prove the theorem, let us begin by slightly extending Ginibre's inequalities \cite{Gin} to spin models on $\Hstar(\mathbb{S})$ (the original inequality deals with $H_{\exact}(\mathbb{S})$).
\begin{lemma}[Ginibre] \label{L: Ginibre}
		Consider the setup as in Section \ref{sec:setup2}. 
		For all positive definite functions $F:\mathbb{S} \to \RR$ and all $e, f \in E$, we have
		\[
			\frac{\partial}{\partial \beta_e}\mu_{\beta, \coclosed}[F(\J_f)]\geq 0. 
		\]
\end{lemma}
\begin{proof}
	This is proved in Appendix \ref{ap:Gin}. 
\end{proof}

\begin{proof}[Proof of Theorem \ref{T: var_increasing_height}]
We first add an additional edge $g$ connecting $x$ and $y$ (even if there was already such an edge present).
On this edge, we put the potential $-\c{U}_{g}(t) = \cos(t)$ and parameter $\beta_{g} =\lambda \geq 0$. 
Thus, we remain in the setup of Section~\ref{sec:setup2}. We write $\mu_{\beta,\lambda,\coclosed}$ and $\nu_{\beta,\lambda,\exact}$ for the corresponding spin and height-function measure respectively, and note that $\mu_{\beta,\lambda,\coclosed}\to \mu_{\beta,\coclosed}$ as $\lambda\to \infty$. 

Let $\epsilon$ be any $1$-form vanishing on $g$ and such that $\d^*\epsilon=\delta_x-\delta_y$, and let $\epsilon'$ be the $1$-form vanishing outside of $g$ and such that $\d^*\epsilon'=\delta_x-\delta_y$.
By Lemma \ref{L: covariance} applied first to $\epsilon'$ and then to $\epsilon$, we have that 
\begin{align}\label{eq:incr}
	\mathbb \nu_{\beta,\lambda, \exact}[(h_x-h_{y})^2] &=  \beta_g\mu_{\beta,\lambda, \coclosed}[\cos(\J_g)] +\beta_g^2\mu_{\beta, \lambda,\coclosed}[\cos(\J_g)^2]-\beta_g^2 \nonumber \\
	&= \sum_{e \in E}\big( \mu_{\beta,\lambda, \coclosed}[\c{U}_e''(\J_e)] \epsilon^2_e -  \mu_{\beta,\lambda, \coclosed} \big[(\c{U}_e'(\J), \epsilon)^2\big]\big) .
\end{align}
Since $2\cos^2 t=1+\cos2t$ is positive definite we can apply Lemma \ref{L: Ginibre} to the first line above and conclude that \eqref{eq:incr} is nondecreasing in $\beta_e$ for any $e\neq g$.
By weak convergence, the same holds for 
 \[
\sum_{e \in E}\big( \mu_{\beta, \coclosed}[\c{U}_e''(\J_e)] \epsilon^2_e -  \mu_{\beta, \coclosed} \big[(\c{U}_e'(\J), \epsilon)^2\big]\big) = \mathbb \nu_{\beta, \exact}[(h_x-h_{y})^2],
 \]
where the last equality again follows from Lemma \ref{L: covariance}. This ends the proof. 
\end{proof}

%% file: sections/delocalization.tex

\subsection{Delocalisation of roughly planar height function models.}
In this section, we will use Theorem \ref{T: var_increasing_height} to deduce that on many planar graphs, the height function delocalises. 
Consider here an infinite lattice $\Gamma = (\mathscr{V}, \mathscr{E})$ embedded in the plane, 
but not necessarily planar. We will assume throughout that $\Gamma$ (under this embedding) invariant under a bi-periodic lattice action, 
and that it has finite degrees. 
An example of such $\Gamma$ is $\ZZ^2$ where all vertices are connected if they are within distance $M < \infty$ from each other. 
Given $\Gamma$, recall that $(G_N)_{N \geq 1}$ is an exhaustion of $\Gamma$ if $G_N$ is a finite subgraph of $\Gamma$ for each $N$, $G_N \subset G_{N + 1}$ and $G_N \nearrow \Gamma$. 
We will also consider the \emph{wiring} of $G_N$ by identifying $G_N^c$ in $\Gamma$ to a single vertex $\partial$ and removing all the self-loops created in this process. 
The obtained graph will be denoted by $G_N^*$. 
On such graphs, we will take the measures $\nu_{N, \beta, \exact}$ as in Section \ref{sec:setup2}, and we identify the space of $1$-forms $\n$ in $\Hcycle(\mathbb{Z})$ with functions $h$ in $\Omega^0_0(\ZZ)$. 

\begin{theorem}[Delocalisation] \label{T: deloc}
	Let $\Gamma$ be as above and consider the setup as in Section~\ref{sec:setup2} where we assume that $\c{U}_e$ is the same for all edges. 
	There exists a $\beta_c < \infty$ such that for all $\beta \geq \beta_c$ and all wired exhaustions $G_N^* \nearrow \Gamma$, 
	\[
		\nu_{N, \beta, \exact}[h_o^2] \to \infty. 
	\]
\end{theorem}

To prove this theorem, we rely on a beautiful result of Lammers \cite{Lammers}:
\begin{theorem}[Theorem 2.7 \cite{Lammers}] \label{T: Lammers}
	Let $\Gamma' = (V, E)$ be an infinite graph with degree at most three, that is invariant under some lattice action. 
	If $\c{V}$ is a convex potential for the height function with
	\[
		\c{V}(\pm 1) \leq \c{V}(0) + \log(2), 
	\]
	then the height function delocalises in the sense that there are no translation invariant Gibbs measures. 
\end{theorem}

In general, the potentials $\c{V}$ as in the setup of Section \ref{sec:setup2} need \emph{not}  be convex. 
However, in some special cases they are, as we will show next.  
This will be crucial for what follows: in Section \ref{sec: red to convex} it will be shown that we can always reduce to this case. 

\begin{lemma} \label{L: XY convex deloc}
	If $-\c{U}(\J) = \cos(i\J)$ for some $i \in \NN$, then $\c{V}_{\beta}$ is convex over $i\ZZ$ for all $\beta$. 
	Moreover, translation invariant Gibbs measures exist if and only if $\nu_{N, \beta, \exact}[h_o^2]$ is bounded uniformly in $N$. 
\end{lemma}
\begin{proof}
	In the case $-\c{U}(\J) = \cos(\J)$, convexity of $\c{V}_{\beta}$ over the integers is an easy consequence of the Tur\'an inequality, see e.g.~\cite{vEnLis}. 
	The extension to $-\c{U}(\J) = \cos(i\J)$ follows from a change of variables. 
	The second statement of the lemma was proved in the case of the XY model in \cite[Theorem 4]{vEnLis}. 
	It follows from a standard dichotomy (see e.g. \cite{LamOtt}) in the case where the height function satisfies the so called ``absolute value FKG'' property, 
	meaning that $|h|$ is FKG, see also \eqref{eq:vardeloc}. 
\end{proof}

\begin{remark}
	We wish to point out that the result of Lammers does \emph{not} depend on the potential $\c{V}$ being the same on each edge, just that it satisfies the condition of Theorem~\ref{T: Lammers} for all edges, and that the potentials are invariant under some lattice action. 
\end{remark}

We will first modify the potentials $-\c{U}$ so they will fit the framework of Theorem \ref{T: Lammers} and Lemma \ref{L: XY convex deloc}. 
Next, we modify the graph $\Gamma$ to obtain a graph $\Gamma'$ to which we can apply Theorem \ref{T: Lammers}
in such a way that $\Gamma'$ embeds into $\Gamma$ and the variance of the height function in $\Gamma'$ is smaller. 

\subsubsection{Reduction to convex potentials} \label{sec: red to convex}
We will apply here a simplification that allows us to only consider potentials of the form $-\c{U}(\J) = \alpha_i \cos(i\J)$. 
Since $-\c{U}$ is positive definite, it can be written as
\[
	-\c{U}(\J) = \alpha_0 + \sum_{i = 1}^\infty \alpha_i\cos(i \J),
\]
with $\alpha_i \geq 0$. Now let $i \geq 1$ be the first mode where $\alpha_i > 0$. Write $-\c{U}'= \alpha_i \cos(i\J)$ 
and $\nu'_{G, \beta, \exact}$ for the corresponding height function measure. 

\begin{lemma} \label{L: pure-pot reduction}
	For any finite graph $G = (V, E)$ with boundary $\partial$ and any $x \in V \setminus \{\partial\}$, we have
	\[	
		\nu'_{G, \beta, \exact}[h_x^2] \leq \nu_{G, \beta, \exact}[h_x^2].
	\]
\end{lemma}
\begin{proof}
	Take $-\c{U}'' = -\c{U} + \c{U}'$ which is positive definite. 
	Write for any $\alpha \geq 0$ 
	\[
		\c{U}_{\alpha}(\J) = \c{U}'(\J) + \alpha\c{U''}(\J), 
	\]
	so that $\c{U}_1 = \c{U}$, $\c{U}_0$ = $\c{U}'$ and $-\c{U}_{\alpha}$ is positive definite for each $\alpha$. 
	 Let $\nu_{G, \beta, \alpha, \exact}$ be  the corresponding height function measure. 
	 Theorem \ref{T: var_increasing_height} implies that for any $x \in V$
	 \[
	 	\frac{\partial}{\partial \alpha} \nu_{G, \beta, \alpha, \exact}[h_x^2] \geq 0,
	 \]
	 so that the variance is minimized at $\alpha = 0$. This shows the result. 
\end{proof}

\subsubsection{Graph Modifications.} \label{sec: graph mod}
Fix $\Gamma$ an infinite graph and $G_N \nearrow \Gamma$ an exhaustion as above.   
We wish to perform two operations:
\begin{enumerate}[\hspace{1cm}(a)]
	\item splitting edges into multiple sub-edges and
	\item gluing vertices together, 
\end{enumerate}
in such a way that the variance of the height function does not increase. 

Operation (a) is the easiest: to add $k - 1$ ``evenly spaced'' vertices to an edge without changing the height function on the original graph, 
we wish to find a potential $\c{V}_{\beta}^{(k)}$ such that 
\[
	e^{-\c{V}_{\beta}} = (e^{-\c{V}_{\beta}^{(k)}})^{*k},
\]
where by $^{*k}$ we mean $k$-fold convolution. 

Using basic properties of the Fourier transform, we can take the potential $\c{V}_{\beta}^{(k)} = \c{V}_{\beta / k}$ which is dual to $-(\beta / k)\c{U}$.  
This offers the following lemma. 

\begin{lemma}[Splitting edges] \label{L: splitting edges}
	Suppose $\c{V}$ corresponds to a spin potential $\c{U}$, such that $-\c{U}$ is positive definite.  
	For each $k \in \NN$, we have
	\[
		e^{-\c{V}_{\beta}} = (e^{-\c{V}_{\beta / k} })^{*k}.
	\]
\end{lemma}

Operation (b) will make use of Theorem \ref{T: var_increasing_height}. 
Let $v_1, v_2$ be two vertices in the graph, 
with or without an edge between them and add to the graph the edge $g = \{ v_1, v_2\}$ with the XY potential $\c{V}_{\lambda}(k) = -\log(I_k(\lambda))$ 
with parameter $\lambda$. 
Write $\nu_{N, \beta, \lambda, \exact}$ for the corresponding height function measure on $G_N^*$. 
We will show now that gluing the vertices $v_1, v_2$ corresponds to sending $\lambda$ to $0$ in this setting. 
Indeed, as $\lambda \to 0$, we have
\[
	e^{-\c{V}_\lambda(k)} = I_k(\lambda) \to \begin{cases}
		1, & \text{if } k = 0\\
		0, &\text{else }
	\end{cases} 
\]
which means that the height function measure $\nu_{N, \beta, 0, \exact}$ is supported on height functions with $h_{v_1} = h_{v_2}$. 
Moreover, Theorem \ref{T: var_increasing_height} implies that for any vertex $x$ of $G_N^*$, 
\[
	\frac{\partial}{\partial \lambda} \nu_{N, \beta, \lambda, \exact}(h_x^2) \geq 0, 
\]
so that we find the following result. 

\begin{lemma}[Gluing vertices] \label{L: Gluing vertices}
	Let $x, v_1, v_2 \in V$ and $H_N^*$ be obtained from $G_N^*$ by gluing together $v_1$ and $v_2$. Then
	\(
		\nu_{H_N, \beta, \exact}[h_x^2] \leq \nu_{G_N, \beta, \exact}[h_x^2]. 
	\)
\end{lemma}

To summarise, we have established that gluing two vertices together reduces the variance of the height function, and splitting edges as in Lemma \ref{L: splitting edges} does not change the model. 
These two facts together imply that we can modify $\Gamma$ to obtain a planar graph $\Gamma'$ of degree at most three as we will explain now. 
We first show how to go from any planar graph to a planar graph of degree at most three. 

\subsubsection*{Star-tree transform}
There are many ways to transform a planar graph into a planar graph with degree at most three.
We follow here the elegant idea presented in \cite{GurNach}, where it was (implicitly) stated for the Gaussian free field. 
 Suppose that $G = (V, E)$ is a planar graph with boundary $\partial \in V$ and take the setup of Section \ref{sec:setup2}. Fix a vertex $v_0 \in V$.
It will be slightly more convenient to make a distinction between the number of neighboring vertices of $v_0$ and its degree in multigraphs.
\\
\\
\noindent \textbf{Degree reduction algorithm at $v_0$.} 
\begin{enumerate}[\hspace{0.4cm}1.]
	\item If the number of neighbours of $v_0$ is strictly less than $4$, do nothing.
	\item Label all neighbors of $v_0$ by $v_1, \ldots, v_{2d}$ by starting somewhere and going clockwise around $v_0$, where we \emph{do not} include the last vertex if the number of neighbours is odd. 
	\item Add to each edge $v_0v_i$ an intermediate vertex $x_i$ (note that if there are multiple edges between $v_0$ and $v_i$, then we have created many new vertices). 
	\item put the potential $\c{V}_{\beta_{v_0v_i} / 2}$ on the edges $v_0x_i$ and $x_iv_i$, for each $i$. 
	\item Glue together each pair $x_{2i - 1}$ and $x_{2i}$ (this includes gluing together multiple vertices $x_i$ if they exist). 
\end{enumerate}

Note that this algorithm reduces the \emph{number of neighbors of} $v_0$ by a factor $2$ if this number is even. 
Also note that it creates a multigraph. 
From the splitting and gluing lemmas, we obtain the next result. 
\begin{lemma}\label{L: star-tree transform}
	Applying the degree reduction algorithm at $v_0$ does not increase the variance of $h_x$ for any $x \in V$. 
\end{lemma}

Thus, to reduce the number of neighbors of $v_0$ to $3$ or less, we are left to apply the reduction algorithm inductively, 
and to get a graph of degree three we apply it to all vertices in $G$ other than the boundary vertex. 

To finalize the star-tree transform, we still need to transform the multi-graph into a simple graph. Of course, 
we need to do so without changing the height function model. 
If $e_1, e_2$ are two edges with the same end-points $x$ and $y$ then
\begin{align} \label{eq:adding potentials}
	\nu_{N, \beta, \exact}(h_x - h_y = k) \propto e^{-\c{V}_\beta(k)} e^{-\c{V}_\beta(k)} = e^{-(\c{V}_{\beta} + \c{V}_{\beta})(k)}. 
\end{align}
This observation implies that applying inductively the reduction algorithm and then applying the above observation does result in a graph where:
\begin{enumerate}[(i)]
	\item the number of neighbors of each vertex is less than or equal to $3$,
	\item the variance of the height function is not increased, 
	\item the potentials are of the form $D\c{V}_{\beta / k}$ for some $D$ and $k$ that can depend on the edges.  
\end{enumerate}

\begin{center}
	\begin{figure}
		\begin{subfigure}{.15\textwidth}
			\centering
			\includegraphics[scale =0.20]{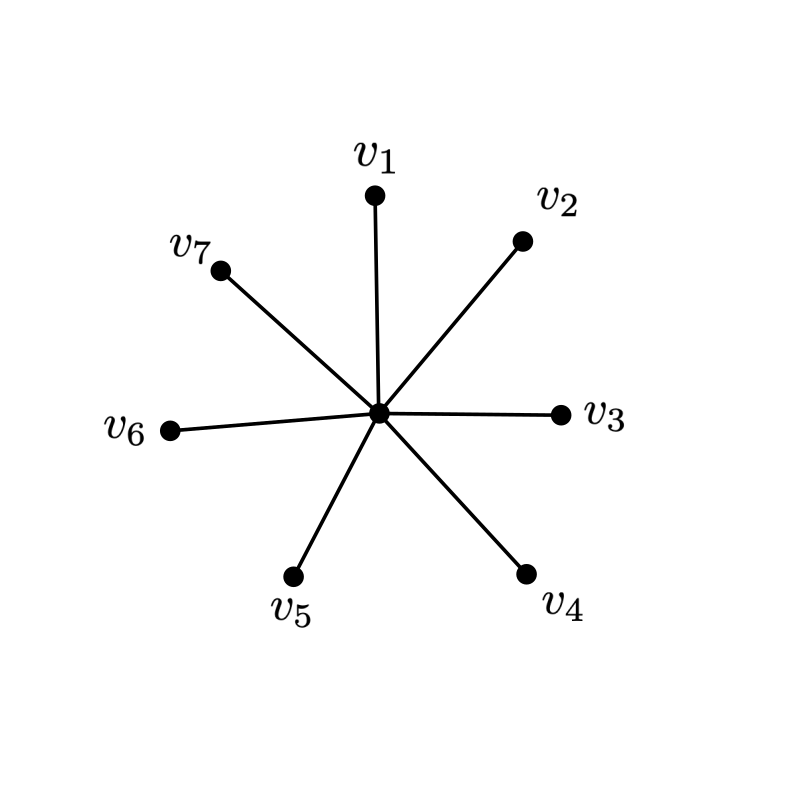}
		\end{subfigure} \hspace{1.0cm}
		\begin{subfigure}{.15\textwidth}
			\centering
			\includegraphics[scale =0.20]{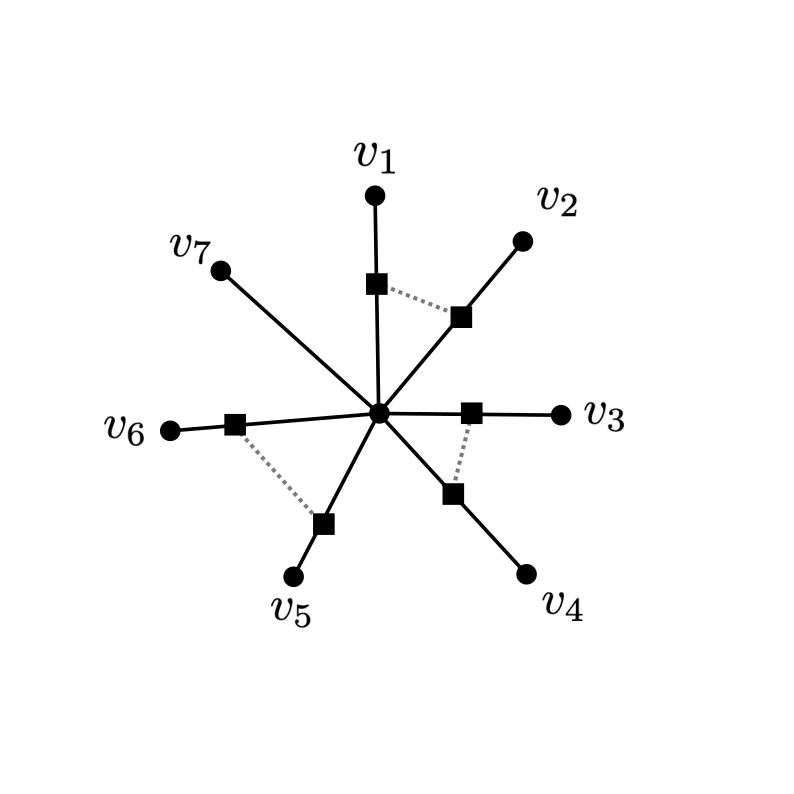} 
		\end{subfigure}  \hspace{1.0cm}
		\begin{subfigure}{.15\textwidth}
			\centering
			\includegraphics[scale =0.20]{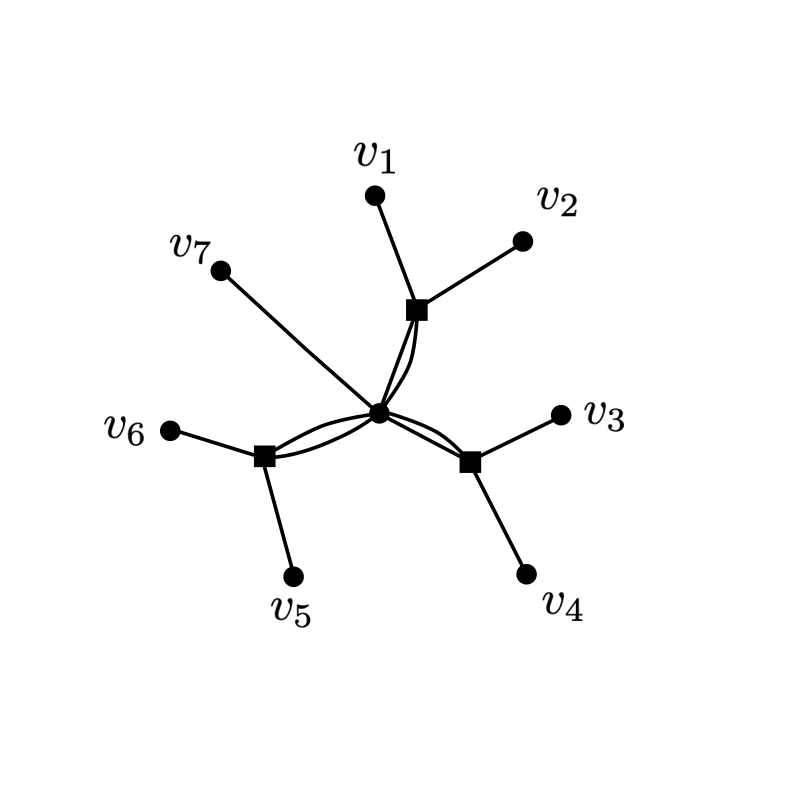} 
		\end{subfigure} \hspace{1.0cm}
		\begin{subfigure}{.15\textwidth}
			\centering
			\includegraphics[scale =0.20]{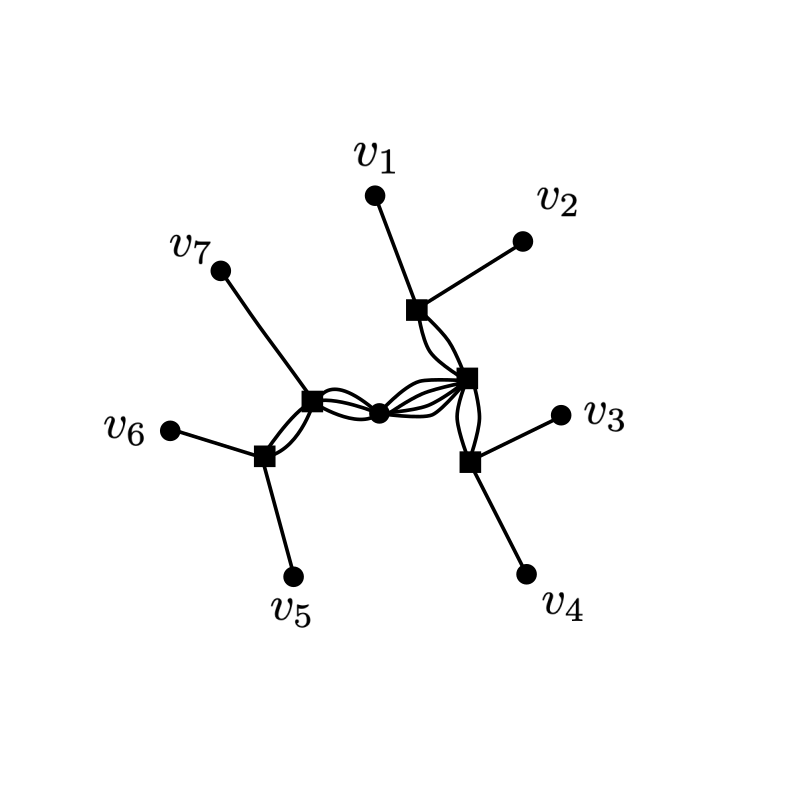} 
		\end{subfigure} \hspace{1.0cm}
		\caption{Left: original graph with $v_0$ in the center. 
			Middle two: first step of the degree-reduction algorithm, dotted lines correspond to vertices to be glued together. 
			Right: Final graph after ``star--tree'' transform, with vertices glued together and all vertices have three or less neighbors.}
		\label{f:star-tree}
	\end{figure}
\end{center} 

\subsubsection{Proof of Theorem \ref{T: deloc}}
Before we finish the proof of Theorem \ref{T: deloc}, let us briefly mention how to go from a ``roughly planar'' graph to a planar graph, see also Figure \ref{f:long_range}. 
We will do so for $\ZZ^2$ where $x \sim y$ if $|x - y|_2 \leq 2$. 
Add to an edge connecting two vertices $x$ and $y$ that are at distance $2$ from eachother a new vertex. 
Glue it to the unique vertex between $x$ and $y$ that is at distance $1$ of each. 
Apply this algorithm to all edges. The obtained graph is planar and the variance of the height function is not increased by Lemma \ref{L: Gluing vertices}. 

\begin{proof}[Proof of Theorem \ref{T: var_increasing_height}]
	Consider the setup of Section \ref{sec:setup2}. 
	By Lemma \ref{L: pure-pot reduction}, we can assume without loss of generality that $-\c{U}(\J) = \cos(i\J)$.  
	Write $\c{V}_{\beta}$ for the corresponding height function potential. 
	
	Let $\Gamma'$ be the planar graph obtained from $\Gamma$ as in Section \ref{sec: graph mod}, with the corresponding potentials $D_e\c{V}_{\beta / k_e}$, 
	$D_e \in (0, \infty)$ and $k_e \in \NN$. 
	Although $D_e, k_e$ may be different on distinct edges, they are uniformly bounded because $\Gamma$ (and hence $\Gamma'$) is invariant a bi-periodic lattice action. 
	
	Let $(G_N)_{N \geq 1}$ be any exhaustion of $\Gamma$ and let $(G_N')_{N \geq 1}$ be the induced exhaustion of $\Gamma'$, obtained from applying the degree reduction algorithm to all of $G_N$ (but not the boundary vertex). 
	Write $\nu'_{N, \beta, \exact}$ for the corresponding height function measure on $G_N'$. 
	It follows from Section \ref{sec: graph mod} that it suffices to prove that for all $\beta$ large enough, $\nu'_{N, \beta, \exact}[h_o^2] \to \infty$. 
	Indeed, in this case we also have $\nu_{N, \beta, \exact}[h_o^2] \to \infty$. 
	
	Note that since $\c{V}_{\beta / k}$ is convex, so is any multiple. 
	Moreover, for each $D \in (0, \infty)$ and $k \in \NN$ we have that for all $\beta$ large enough, 
	\[
		D\c{V}_{\beta / k}(0) \leq D\c{V}_{\beta / k}(1) + \log(2). 
	\]
	Indeed, this follows from the fact that the modified Bessel function satisfies $I_{m}(\beta) / I_{m'}(\beta) \to 1$ as $\beta \to \infty$ (see e.g.~\cite{vEnLis}). 
	Therefore, we can apply Theorem \ref{T: Lammers} and Lemma \ref{L: XY convex deloc} to deduce that for $\beta$ large enough, 
	\[
		\nu'_{N, \beta, \exact}[h_o^2] \to \infty
	\]
	as $N \to \infty$. 
	This proves the theorem. 
\end{proof}

\begin{center}
	\begin{figure}
		\begin{subfigure}{.20\textwidth}
			\centering
			\includegraphics[scale =0.20]{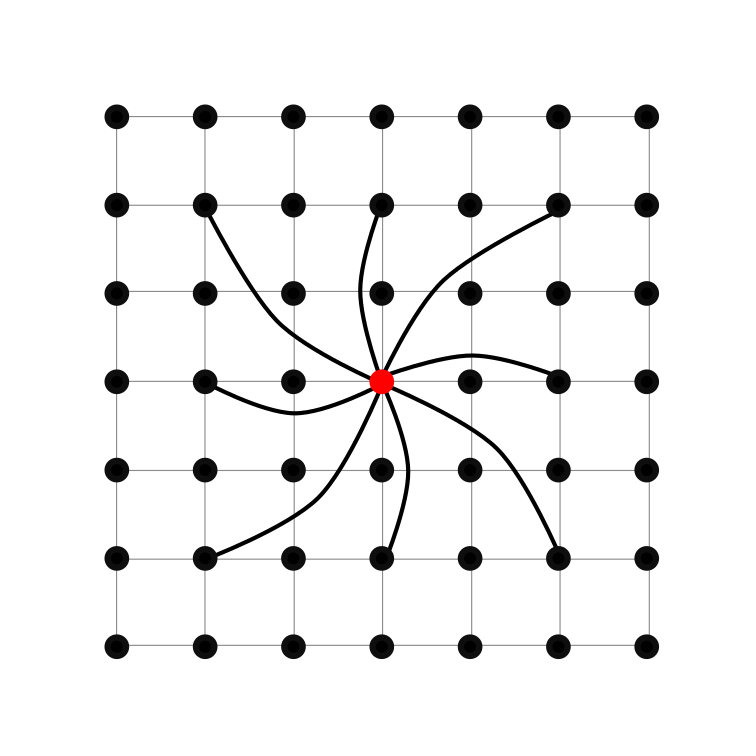}
		\end{subfigure} \hspace{1.5cm}
		\begin{subfigure}{.20\textwidth}
			\centering
			\includegraphics[scale =0.20]{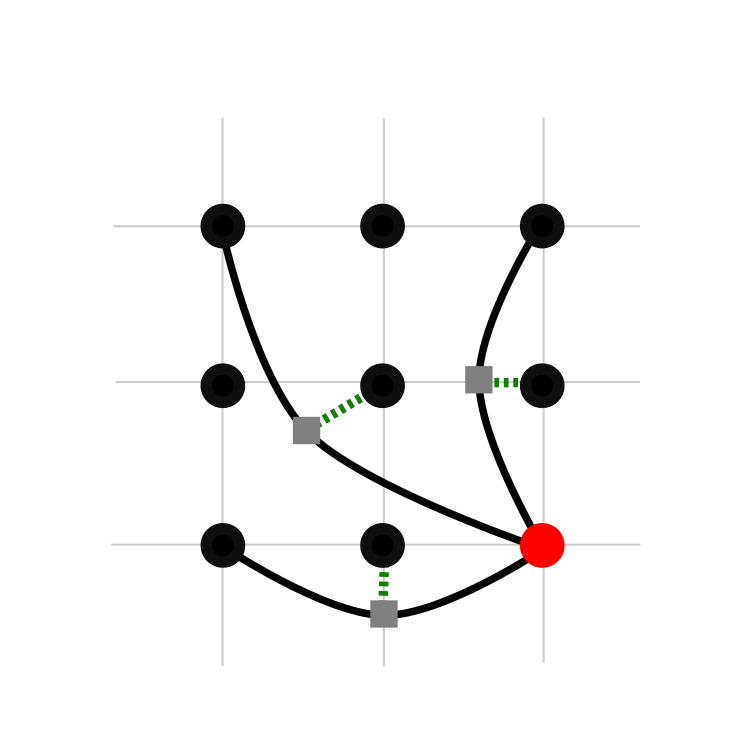} 
		\end{subfigure}  \hspace{1.5cm}
		\begin{subfigure}{.20\textwidth}
			\centering
			\includegraphics[scale =0.20]{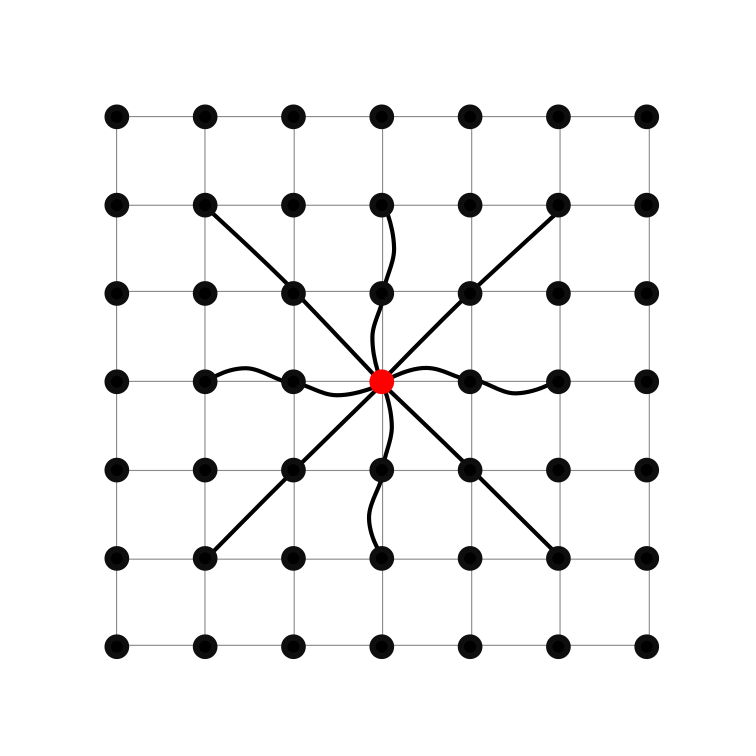} 
		\end{subfigure} \hspace{1.0cm}
		\caption{Left: an example of $\ZZ^2$ with long-range interactions; only the edges of the (red) origin are drawn. Middle: gluing. The square (gray) vertices are added, together with the green edges where the gluing will happen. Right: the final (planar) graph.}
		\label{f:long_range}
	\end{figure}
\end{center} 

%% file: sections/remainder.tex

\section{GFF covariance for a projection of the spin model.} \label{sec:projection}
Let $G = (V, E)$ be a finite graph. Recall that $\Omega^1(\RR)$ equipped with the $l^2$-inner product
\(
	(\epsilon, \omega)_{\Omega^1} 
\)
is a Hilbert space. In this setting, the linear operators $\d$ and $\d^*$ are mutually adjoint, and hence
the spaces $\Hcycle(\RR)$ and $\Hstar(\RR)$ are orthogonal in $\Omega^1(\RR)$ and span the whole space, i.e.,
\[
\Omega^1(\RR) = \Hcycle(\RR) \oplus \Hstar (\RR). 
\]
We denote by $P_{\coclosed}$ and $P_{\exact}$ the orthogonal projection onto $\Hstar(\RR)$ and $\Hcycle(\RR)$ respectively. 

We focus on finite graphs $G = (V, E)$ with boundary vertex $\partial \in V$ and 
take mutually dual potentials $\c{V}$ and $\c{U}$ as in Definition \ref{def:potential}. 

Since $\c{U}$ is symmetric around $0$ by assumption, the derivative $\c{U}'$ of $\c{U}$ is odd and hence $\c{U}'(\J)$ is a $1$-form in $\Omega^1(\RR)$. 
It thus makes sense to look at the orthogonal decomposition of $\c{U}'(\J)$ in the space $\Hcycle (\mathbb R) \oplus \Hstar (\mathbb R)$. 
Define $\tau$ to be the unique element of $\Omega^0_0(\mathbb R)$ such that 
\[
	\d \tau = P_{\exact}\c{U}'(\J).
\]
We will next obtain the -- in our eyes somewhat remarkable -- result that $\tau$ has the covariance of a Gaussian free field 
irrespective of $\c{U}$. 

\begin{proposition}[GFF covariance] \label{P: projected-Gaussian}
	Let $\tau$ be as above and $ f,g \in \Omega^0_0(\RR)$. Then
	\[
			\inf_{e \in E} \mu_{G, \exact}[\c{U}''(\J_e)] (f, \Green g) \leq \mu_{G, \exact}[(\tau, f)(\tau, g)] \leq \sup_{e \in E} \mu_{G, \exact}[\c{U}''(\J_e)] (f, \Green g). 
	\]
\end{proposition}

We begin with an easy consequence of the duality lemma for covariance \ref{L: covariance}.

\begin{lemma}\label{L: spin Guassian dom}
	For any $f,g\in \Omega^0_0(\mathbb R)$, we have	
	\[
		\mu_{G, \exact}[(\c{U}'(\J), \d g)(\c{U}'(\J), \d f)] = \sum_{e \in E}\mu_{G, \exact}[\c{U}''(\J_e)]\d f_e \d g_e
	\]
\end{lemma}

\begin{proof}
	Let $f,g$ be as in the statement so that $\d f, \d g\in \Hcycle$. 
	Lemma \ref{L: covariance} implies
	\[
		\mu_{G, \exact}[(\c{U}'(\J), \d f)(\c{U}'(\J), \d g)] = \sum_{e \in E} \mu_{G, \exact}[\c{U}''(\J_e)]  \d f_e \d g_{e} - \nu_{G, \coclosed}[(\n, \d f)(\n, \d g)]
	\]
	Since $\Hcycle$ and $\Hstar$ are orthogonal, and the dual height $1$-form $\n$ takes value in $\Hstar$, the right-most term vanishes and the result follows. 
\end{proof}

\begin{proof}[Proof of Proposition \ref{P: projected-Gaussian}]
	For any $g \in \Omega^0(\mathbb R)$, we have 
	\[
		(\c{U}'(\J), \d g) = (P_{\exact}\c{U}'(\J), \d g) = (\tau, \Delta g)
	\] 
	where the first equality holds because $\d g \in \Hcycle$ and the second by definition of $\tau$, the duality of $\d$ and $\d^*$ and because $\Delta = \d^*\d$. 
	As before, $\Green$ denotes the inverse of $\Delta$ (so defined that functions take value $0$ at the boundary) and take now $g = \Green m$, $h = \Green f$. 
	It follows from this and the previous lemma that 
	\[
		\mu_{G, \exact}[(\tau, m)(\tau, f)] =  \sum_{e \in E}\mu_{G, \exact}[\c{U}''(\J_e)] \d f_e \d g_e, 
	\]
	implying the desired result. 
\end{proof}

\section{A central limit theorem} \label{sec:CLT}
Here we consider the same setup as in Section~\ref{sec:setup}, and we establish a central limit theorem for $\mathcal U'(J)$ summed over the path~$p_n$.
The main conclusion of this section is that even though the decay of the correlations of $\mathcal U'(J)$ changes if the height function delocalises, $\mathcal U'(J)$ always satisfies a central limit theorem as shown below. 

Let $(N_k)_{k\geq1}$ be a sequence along which $\mu_{\mathbb{T}_N, \coclosed}$ converges weakly to a measure $\mu=\mu_{\mathbb{T}_N}$. As usual, by duality, $\mu$ can be thought of as a Gibbs measure on $H_\exact(\mathbb Z^2, \mathbb S)$.
By Remark~\ref{rem:tightness}, for any fixed $n$, the difference $h_{v_0}-h_{v_n}$ converges weakly under $\nu_{\mathbb{T}_N, \exact}$, as $N\to \infty$ (as long as $v_0,v_n\in \mathbb{T}_N$). Moreover by Corollary~\ref{cor:upper},
\[
	\lim_{n\to \infty} \limsup_{k\to \infty} \nu_{\mathbb{T}_{N_k}, \exact} [(h_{v_0}-h_{v_n})^2]/{n} \leq \lim_{n\to \infty}c \frac{ \log n }{n}= 0
\] 
for some $c<\infty$, and hence by Lemma~\ref{lem:duality}, for all $t\in \mathbb R$,
\begin{align}\label{eq:phi1}
	1= \lim_{n\to \infty} \lim_{k\to \infty} \nu_{\mathbb{T}_{N_k}, \exact}\big [\exp \big(i\tfrac t {\sqrt n} ({h_{v_0}-h_{v_n}})\big)\big]= \lim_{n\to \infty}\lim_{k\to \infty}{Z_{\mathbb{T}_{N_k},\coclosed}\big (\tfrac t  {\sqrt n}p_n\big)}/{Z_{\mathbb{T}_{N_k},\coclosed}(0)},
\end{align}
where again we identify the path $p_n$ with the associated 1-form.
Using that 
\[
\mathcal U(J+\varepsilon)-\mathcal U( J)= \mathcal U'(J)\varepsilon +\tfrac 12 \mathcal U''(J)\varepsilon^2 +o(\varepsilon^2)
\]
we can write
\begin{align*}
{Z_{\mathbb{T}_{N_k},\coclosed}\big (\tfrac t  {\sqrt n}p_n\big)}/{Z_{\mathbb{T}_{N_k},\coclosed}(0)} 
&= \mu_{\mathbb{T}_{N_k}, \coclosed} \Big [\exp \Big(-t\frac{1}{\sqrt n}\sum_{i=1}^n  \mathcal U' (J_i)  - {t^2} \frac {1}{2n} \sum_{i=1}^n \mathcal U''(J_i)+o(t^2)\Big)\Big],
\end{align*}
where the error is uniform in $k$.
We first note that by weak convergence, as $k\to \infty$, the right-hand side approaches to the same expectation but with respect to $\mu$ (the infinite volume limit of $\mu_{\mathbb{T}_N, \exact}$).
Moreover, the error therm vanishes in the limit $n\to\infty$. Assuming that the spin measure $\mu$ is ergodic, we also have that
\[
\frac 1n \sum_{i=1}^n \mathcal U'' (J_i)\to \mu[\mathcal U'' (J_0)] \quad \mu\textnormal{-a.s. as } n\to \infty
\] 
by Birkhoff's pointwise ergodic theorem (since $J$ is invariant under the shift along the path, and $\mathcal U''$ is bounded). 
Note that in the case of the XY model, there is only one translation invariant Gibbs measure in two dimensions \cite{MMSP} which must therefore be ergodic. 
By the dominated convergence theorem and \eqref{eq:phi1}, we conclude the following central limit theorem.
\begin{theorem} If $\mu$ is ergodic, then for any $t\in \mathbb R$,
\begin{align*}
\lim_{n\to \infty} \mu \Big [\exp \Big(-\frac{t}{\sqrt n}\sum_{i=1}^n  \mathcal U' (J_i)\Big)\Big] = \exp\Big(\frac{t^2}2   \mu [\mathcal U'' (J_0)] \Big).
\end{align*}
In particular,
\[
\frac1{\sqrt n}\sum_{i=1}^n  \mathcal U' (J_i)\to \mathcal N(0, \mu [\mathcal U'' (J_0)])
\]
in distribution as $n\to \infty$.
\end{theorem}

%% file: sections/appendix.tex

\section{Duality between height functions and spin models} \label{ap:duality}
Let $G = (V, E)$ be a finite graph and write $\vec{E}$ for its oriented edges. Here and in what follows, we will always fix an implicit embedding $E \hookrightarrow \vec{E}$, 
which fixes for each edge a prescribed orientation. The particular embedding chosen is of no importance. 

Recall that the two Hilbert spaces $\Omega^0(\RR), \Omega^1(\RR)$ are equipped with the natural inner products:
\[
(f, g)_{\Omega^0} := \sum_{x \in V} f_x g_x \quad \text{ and } \quad 
	(\epsilon, \omega)_{\Omega^1} := \frac{1}{2}\sum_{\vec{e} \in \vec{E}} \epsilon_{\vec{e}} \:\omega_{\vec{e}}
\]
respectively, and it is standard to see that $\d$ and $\d^*$ are adjoints: for all $f \in \Omega^0(\RR)$ and $\omega \in \Omega^1(\RR)$
\[
	(\d f, \omega)_{\Omega^1} = (f, \d^* \omega)_{\Omega^0}.
\]
We will use the embedding $\mathbb{S} \hookrightarrow \RR$ uniquely defined by identifying $e^{i \theta} \leftrightarrow \theta$ 
in such a way that $\theta \in (-\pi, \pi]$.

Recall the definition for any $1$-form $\epsilon \in \Omega^1(\RR)$ of the twisted partition function
\[
	Z_{\#}(\epsilon) = \int_{H_{\#}(G, \mathbb{S})} \prod_{e \in E} w(\J_e + \epsilon_e)  d\J.
\]
We wish to emphasize again that when $\# = \coclosed$, unlike in the case of planar graphs, this does not generally correspond to a spin model on vertices, 
but rather to a measure on $1$-forms taking value in the group $\mathbb{S}$ and satisfying $\d^* \J = 0$.

We will need an appropriate description of the Haar measure on $H_{\#}(G, \mathbb{S})$. 
Let us start with the case $\#  =  \exact$. 

\begin{lemma} \label{Lapp: Hcycle def}
	For any function $F: \Omega^1(G, \mathbb{S}) \to \RR$
	\[
		\int_{\Hcycle(G, \mathbb{S})} F(\J) d \J = \int_{\mathbb{S}^V} F(\d \theta) d \theta
	\]
	where $d\theta$ is the product uniform measure on $\mathbb{S}^V$. 
\end{lemma}

\begin{proof}
	We first note that the measure $\nu$ on $\Hcycle(E, \mathbb{S})$ defined via
	\[
		\nu(A) := \int_{\mathbb{S}^V} \id_{A}(\d \theta) d \theta
	\]
	is a Radon probability measure. We are left to argue that it is invariant under the group action, since then it is the unique Haar measure. 
	To that end, let $\J' \in \Hcycle$ and recall that we use additive notation for abelian groups. Note that $\J' = \d \theta'$ for some $\theta' \in \mathbb{S}^V$, so that $\J'+\d \theta= \d\theta' +\d \theta = \d(\theta' +\theta)$. 
	In particular, 
	\[
		\nu( A-J') = \int_{\mathbb{S}^V} \id_{A-J'}(\d \theta) d \theta = \int_{\mathbb{S}^V} \id_{A}(\d (\theta'+ \theta)) d \theta = \nu(A),
	\]
	where the last equality follows as $d \theta$ is the product Lebesgue measure and hence invariant under rotations (translations) of each of the coordinates. 
\end{proof}

For the case $\# = \coclosed$, we will need a different argument. 
An element $\J \in H_{\coclosed}(\mathbb{S})$ satisfies $d^* \J \equiv 0$ by definition. 
Therefore, knowing the value of $\J$ at all edges containing a vertex $x$ but one, 
uniquely determines the value of $\J$ on the last edge. 
Let $T\subset E$ be a spanning tree of $G = (V, E)$ (the exact choice does not matter). Let $G_T=(V,E\setminus T)$.
If $\partial \in V$ is a chosen root, then $T$ can be oriented towards the root and as such, 
each vertex in $V \setminus \{\partial\}$ satisfies that there is exactly one edge in the oriented tree pointing out of $x$. 
Therefore, for each $\J\in \Omega^{1}(G_T, \mathbb{S})$, there is a unique way to extend $\J$ to ${E}$ in such a way that
$\J \in \Hstar(G, \mathbb{S})$, and we will write $\bar{\J}$ for this extension. 

\begin{lemma} \label{Lapp: Hstar def}
	For any function $F: \Omega^1(G, \mathbb{S}) \to \RR$, we have
	\[	
		\int_{\Hstar(\mathbb{S})} F(\J) d \J = \int_{\mathbb{S}^{E \setminus T}} F(\bar{\J}) d \J, 
	\]
	where the measure on the right-hand side is the product uniform measure on $\mathbb{S}^{E \setminus T}$. 
	In particular, the right-hand side does not depend on $T$. 
\end{lemma}
	
\begin{proof}
	Define the measure $\nu$ on $\Hstar(G, \mathbb{S})=\ker(\d^*)$ through $\nu(A) := \int_{\mathbb{S}^{E \setminus T}} \id_{A}(\bar{\J}) d \J$.
	It is enough to show that $\nu$ is invariant under the group action. 
	Indeed, for any $\tau \in \mathbb{S}^{E \setminus T}$, we have 
	\[
	\nu( A-\bar\tau) = \int_{\mathbb{S}^{E \setminus T}} \id_{A}(\overline{\J+\tau}) d \J = \int_{\mathbb{S}^{E \setminus T}} \id_{A}(\bar{\J}) d \J=\nu(A),
	\]
	since the product uniform measure is invariant under the group action. This ends the proof. 
\end{proof}

We will also need the following classical results from Fourier series theory. For proofs, see e.g.~\cite{Bash} or \cite[Theorem IV.2.9]{Wer_FA} (in German).
\begin{lemma} \label{lem:fform}
	 Let $f: \mathbb S\to \mathbb R$ be continuously differentiable. 
	 Then $f(\theta)=\lim_{K\to \infty} f_K(\theta)$, with
	 \[
	 f_K(\theta)=a_0 + \sum_{k=1}^K (a_k e^{i k \theta} + a_{-k} e^{-i k \theta} ), \qquad \text{where} \qquad a_k=  \int_{\mathbb S} e^{- i k \theta}f(\theta){d} \theta.
	 \]
	 Moreover, the convergence is uniform on $\mathbb S$. Finally
	 \[
	  f_K(\theta) =  \int_{\mathbb S} \Big(\sum_{k=-K}^Ke^{ik (\theta'-\theta)}\Big) f(\theta' ){d}  \theta'.
	 \]
\end{lemma}

\begin{proof}[Proof of Lemma \ref{lem:duality}]
	\textbf{Case I: $\nu_{\coclosed}$.} 
		It follows from condition~\ref{eq:Vsum} and from the dominated convergence theorem that for any $\epsilon \in \Omega^1(\mathbb R)$, we have
	\begin{align*}
		Z_{\exact}(\epsilon) &= \int_{H_{\exact}(\mathbb{S})} \prod_{e \in E} w(\J_e + \epsilon_e) d \J \\
		&= \int_{H_{\exact}(\mathbb{S})} \sum_{\n: E \to \ZZ} \prod_{e \in E} e^{i \n_e(\J_e + \epsilon_e)} \exp(-\c{V}(\n_e)) d \J \\
		&= \sum_{\n: E \to \ZZ} e^{i (\n, \epsilon)}\prod_{e \in E}\exp(-\c{V}(\n_e)) \int_{H_{\exact}(\mathbb{S})} e^{ i (\n, \J)} d \J. 
	\end{align*}
	Moreover, by Lemma~\ref{Lapp: Hcycle def} we have
	\begin{align*}
		\int_{\Hcycle(\mathbb{S})} e^{i(\n, \J)} d \J = \int_{\mathbb{S}^V} e^{i(\n, \d \theta)} d \theta 
		= \int_{\mathbb{S}^V} e^{i(\d^*\n, \theta)} d \theta = \prod_{x \in V} \int_{\mathbb{S}} e^{i \d^*\n_x\theta_x} d \theta_x = \mathbf 1\{ \n \in H_{\coclosed} (\mathbb Z)\},
	\end{align*}
	which ends the proof of case I.
	
	\textbf{Case II: $\nu_{\exact}$.} We will show equality of partition functions with $\epsilon =0$, and the general case follows the same steps. By Lemma~\ref{lem:fform} we have
	\begin{align*}
		\prod_{e \in E} \exp(-\mathcal V(\d h_e)) &=  \prod_{e \in E} \int_{\mathbb S} e^{ i \d h_e \theta_e}w(\theta_e){d} \theta_e \\
		&= \int_{\mathbb S^E}  e^{ i (\d h, \theta)_{\Omega_1}} \prod_{e \in E}w(\theta_e)\prod_{e \in E}{d} \theta_e \\
		&= \int_{\mathbb S^E}  e^{ i (h, \d^*\theta)_{\Omega_0}} \prod_{e \in E}w(\theta_e)\prod_{e \in E} {d}\theta_e,	
	\end{align*}
	and hence
	\begin{align}
	&\mathop{\sum_{h: V\to \mathbb Z}}_{h_\partial =0} \prod_{e \in E} \exp(-\mathcal V(\d h_e))  = \lim_{K_{v_n}\to \infty} \cdots  \lim_{K_{v_1}\to \infty}\nonumber
	 \sum_{h_{v_n}\in I_{K_{v_n}}}\cdots \sum_{h_{v_1}\in I_{K_{v_1}}} \prod_{e \in E}\exp(-\mathcal V(\d h_e)) \nonumber \\
	 &=  \lim_{K_{v_n}\to \infty} \cdots  \lim_{K_{v_1}\to \infty} \sum_{h_{v_n}\in I_{K_{v_n}}}\cdots \sum_{h_{v_1}\in I_{K_{v_1}}} \int_{\mathbb S^E}  e^{i (h, \d^*\theta)_{\Omega_0}} \prod_{e \in E}w(\theta_e)\prod_{e \in E} {d}\theta_e, \label{eq:bigone}
	\end{align}
	where $I_K=\{-K,\ldots,K\}$, and $v_1, v_2, \ldots, v_n$ is any ordering of $V\setminus \{\partial\}$ that explores the tree $T$ from the leaves towards the root $\partial$.
	
	We will now evaluate the expression above with the use of Lemma~\ref{lem:fform} by iteratively (over $i$) exchanging the order of summation of $h_{v_i}$ with the integration over $\theta_{e_{v_i}}$, and then takin the $K_{v_i}\to \infty $ limit.
	To this end, we orient each edge in $T$ towards the root vertex $\partial$, and to each vertex $v\neq \partial$ we assign the unique outgoing edge $e_v$ from $v$.
	
	In the first step we choose the leaf vertex $v=v_1$, and write
	\[
		\d^* \theta_v = \sum_{w \sim v} \theta_{wv} = -\theta_{e_v} + \theta_{e_1} + \ldots + \theta_{e_{l}},
	\]
	where $l+1$ is the degree of $v$ in $G$, and $e_1, \ldots, e_l$ are the remaining edges in $E$ incident to $v$ and pointing at $v$. 
	Let $x$ be the other endpoint of the edge $e_{v}$, so that $e_{v} = (v, x)$. 
		Given $h_x \in I_{K_x}$ and $(\theta_e)_{e \in E \setminus \{e_v\}}$ apply Lemma~\ref{lem:fform} (separately to the imaginary and real part) with $f(\theta_{e_v}) := w(\theta_{e_v}) e^{i h_x d^* \theta_x}$ to get
		\[
			\int_{\mathbb S} \Big( \sum_{h_v \in I_{K_v}} e^{i h_v \d^*\theta_v}\Big) w(\theta_{e_v})e^{i h_x d^* \theta_x} {d} \theta_{e_v} = f_{K_v}( \theta_{e_1} + \ldots + \theta_{e_{l}}) \to f( \theta_{e_1} + \ldots + \theta_{e_{l}}), 
		\]
		as $K_v \to \infty$ uniformly in $\theta_{e_1} + \ldots + \theta_{e_l}$. This means that we can take $K_v \to \infty$ inside the integral over $\mathbb{S}^{E \setminus \{e_v\}}$. All in all this removes the variables $h_v$, $K_v$ from~\eqref{eq:bigone}, and $\theta_{e_v}$ is replaced it by $\theta_{e_1} + \ldots + \theta_{e_l}$. Define now $\theta^{(1)}_e = \theta_e$ for $e \in E \setminus \{e_v\}$ and $\theta^{(1)}_{e_v} = \theta_{e_1} + \ldots + \theta_{e_l}$. 
		In other words, after step one, \eqref{eq:bigone} becomes
		\[
			 \lim_{K_{v_n}\to \infty} \cdots  \lim_{K_{v_2}\to \infty} \sum_{h_{v_n}\in I_{K_{v_n}}}\cdots \sum_{h_{v_2}\in I_{K_{v_2}}} \int_{\mathbb S^{E \setminus \{e_{v_1}\}}}  \prod_{w \in V \setminus \{v_1\}} e^{i h_w (\d^*\theta^{1})_w} \prod_{e \in E}w(\theta^{1}_e)\prod_{e \in E \setminus \{e_{v_1}\}} {d}\theta_e. 
		\]
		
		We continue this procedure for edge $e_{v_2}$ where we take the corresponding $f$ ($x$ is replaced by the other endpoint of $e_{v_2}$). 
		In this step we remove the variables $h_{v_2}, K_{v_2}$ and $\theta^{(1)}_{e_{v_2}}$, and replace the latter by $\theta^{(1)}_{e_1} + \ldots + \theta^{(1)}_{e_l}$ (where $l$ depends on $v_2$ now). Define then $(\theta^2_e)_{e \in E}$ through $\theta^{(2)}_e = \theta^{(1)}_e$ on $e \in E \setminus \{e_{v_2}\}$ and 
		$\theta^{(2)}_{e_{v_2}} = \theta^{(1)}_{e_1} + \ldots + \theta^{(1)}_{e_l}$. 
		We iterate the procedure until we have done so for all vertices of $V \setminus \{\partial\}$ and arrive at $\theta^{(n)}$. 
		It is clear that $(\d^* \theta^{(n)})_x = 0$ for all $x \in V \setminus \{\partial \}$, and therefore
		\[
			(\d^* \theta^{(n)})_{\partial} = (\d^*\theta^{(n)}, 1)_{\Omega^0} = (\theta^{(n)}, \d 1)_{\Omega^1} = 0, 
		\]
		so that $\d^* \theta^{(n)}$ vanishes on all of $V$. 
		
		Now let $(J_e)_{e \in E \setminus T} = (\theta_e)_{e \in E \setminus T}$ and define $\bar{J}$ the unique extension to $\Hstar(G, \mathbb{S})$ as before. 
		It is easy to check that $\bar{J}$ equals $\theta^{(n)}$. 
		Therefore, at the end of the iterative procedure, we have that~\eqref{eq:bigone} becomes
		\[
			 \int_{\mathbb S^{E\setminus T}} \prod_{e \in E}w(\bar J_e)\prod_{e \in E \setminus T} {d}J_e=Z_{\coclosed}(0), 
		\]
		where the equality follows from Lemma~\ref{Lapp: Hstar def}. This ends the proof of case II.
\end{proof}

\section{Proof of the Ginibre inequality} \label{ap:Gin}
We focus here only on the case where $\J$ takes values in $\Hstar$, as the other case is just the classical Ginibre inequality \cite{Gin}. 
To be precise, we will prove the following fact, from which Lemma \ref{L: Ginibre} follows immediately. 

\begin{lemma}\label{L: Ginibre 2}
	Consider the setup as in Section \ref{sec:setup2}. Let $F, F': \mathbb{S} \to \RR$ be two positive definite functions. Then 
	\[
		\mu_{\beta, \coclosed}(FF') - \mu_{\beta, \coclosed}(F)\mu_{\beta, \coclosed}(F') \geq 0. 
	\]
\end{lemma}

As in the classical proof by Ginibre \cite{Gin}, we will rely on the following result. 

\begin{lemma} \label{L: Ginibre core}
	For any $n \in \NN$ and $(m_i)_{i = 1}^n \in \ZZ^n$, we have 
	\[
		\int_{\Hstar(\mathbb{S})^2} \prod_{i = 1}^n (\cos(m_i\J_{e_i}))  \pm \cos(m_i\J_{e_i}'))d\J d \J' \geq 0, 
	\]
	where the signs $\pm$ might be different for each $i$.
\end{lemma}

\begin{proof}
	We begin by noticing that for any linear $M: \RR^{E} \to \RR$,
	\begin{align*}
		&\cos(M \J) + \cos(M \J') = 2\cos\left(M \frac{\J - \J'}{2}\right) \cos\left(M\frac{\J + \J'}{2}\right), \quad \text{and}\\
		&\cos( M \J) - \cos(M \J') = 2\sin\left( M \frac{\J - \J'}{2}\right) \sin\left(M \frac{\J + \J'}{2}\right). 
	\end{align*}
	Let $T\subset E$ be a spanning tree of $G = (V, E)$ and recall for $J \in \Omega^1(G_T, \mathbb{S})$ the definition of $\bar{J}$ as in Lemma~\ref{Lapp: Hstar def}. By Lemma~\ref{Lapp: Hstar def} we have
	\[
		\int_{\Hstar(\mathbb{S})^2} \prod_{i = 1}^n (\cos(m_i\J_{e_i}))  \pm \cos(m_i\J_{e_i}'))d\J d \J' = \int_{(\mathbb{S}^{E \setminus T})^2} \prod_{i = 1}^n (\cos(m_i\bar{\J}_{e_i}))  \pm \cos(m_i\bar{\J}_{e_i}'))d\J d \J'. 
	\]
	Now consider $J$ in $\Omega^1(G_T, \RR)$ (via the usual identification of $\mathbb{S}$ with $(-\pi, \pi]$) and define by $A_T J$ the unique extension of $J$ to $\Omega^1(G, \RR)$ so that $J \in \Hstar(G, \RR)$, i.e. so that  $\d^* (A_T J) = 0$ in $\RR$. Notice that $A_T J$ and $\bar{J}$ (seen in $\RR$) are equal on all edges in $E \setminus T$, while on an edge $e_i \in T$, the difference is of the form $2\pi k_{i}$ for some integer $k_i$.
	Since the cosine is $2\pi$-periodic, we notice that each factor where the edge $e_i$ is in $T$ is of the form $\cos(m_i(A_T \J)_{e_i}) \pm \cos(m_i(A_T \J')_{e_i})$. 
	All together, we can thus write
	\[
		\int_{\Hstar(\mathbb{S})^2} \prod_{i = 1}^n (\cos(m_i\J_{e_i}))  \pm \cos(m_i\J_{e_i}'))d\J d \J' = \int_{(\mathbb{S}^{E \setminus T})^2}F\left(\frac{\J + \J'}{2}\right)F\left(\frac{\J - \J'}{2}\right)d \J d \J'
	\]
	for some function $F: \mathbb{S}^{E \setminus T} \to \RR$. 	
	Next, make the change of variables via $\tau_e := \frac{\J_e - \J_e'}{2}$ and $\tau'_{e} = \frac{\J_e + \J'_e}{2}$, so that
	\begin{align*}
		\int_{(\mathbb{S}^{E \setminus T})^2}F\left(\frac{\J + \J'}{2}\right)F\left(\frac{\J - \J'}{2}\right)d \J d \J' &= \int_{ (\mathbb{S}^{E \setminus T})^2} F(\tau)F(\tau') d\tau d\tau' \\
		&= \Big(\int_{ (\mathbb{S}^{E \setminus T})^2} F(\tau)d\tau\Big)^2\geq 0.
	\end{align*}
	This ends the proof.
\end{proof}

Note that the collection of functions $t \mapsto \cos(mt)$, $m \in \ZZ$,
generates the positive cone of (real) positive definite functions. 
Since the integral appearing in Lemma \ref{L: Ginibre core} 
is stable under taking convex combinations, this implies that
for any collection $\{ F_i\}$ of positive definite functions $\mathbb{S} \to \RR$, we have
\[
	\int_{\Hstar(\mathbb{S})^2} \prod_{i=1}^n (F_i(\J_{e_i}) \pm F_i(\J_{e_i}')) d \J d \J' \geq 0.
\]
This last remark is also the content of Propositions 1 and Example 4 of \cite{Gin}. 
From this, Lemma \ref{L: Ginibre 2} can be proved in exactly the same way as in \cite[Proposition 3]{Gin}.

\section{Reflection positivity.} \label{ap:RP}
We recall briefly a condition for potentials to be reflection positive. 
For further reference, see e.g. \cite{Biskup} and \cite{FriVel}. 
Fix the torus $\mathbb{T}_n = (\mathbb{Z} / n\ZZ)^d$ and let $\Theta$ be any reflection
(either through edges or through vertices). 
This naturally splits the torus into two parts $\mathbb{T}_n^+$ and $\mathbb{T}_n^-$. 
Let $\mathscr{U}^{\pm}$ be the set of real-valued functions on $\mathbb{T}_n$ depending only on $\mathbb{T}_n^{\pm}$. 
Then $\Theta$ induces a map $\Theta: \mathscr{U}^{\pm} \to \mathscr{U}^{\mp}$.
We will say that a probability measure $\mu$ on $\mathbb{S}^{\mathbb{T}_n}$ is reflection positive with respect to $\Theta$ if
\vspace{0.2cm}
\begin{enumerate}[\hspace{0.5cm}(a)]
	\item $\mu(g \Theta f) = \mu(f \Theta g)$ for all $f, g \in \mathscr{U}^+$, 
	\item $\mu(g \Theta g) \geq 0$. 
\end{enumerate}
\vspace{0.2cm}
\noindent Although the property (a) is not important for us, it is also the easier part and it holds precisely for all measures that are invariant under the reflection $\Theta$. 
It is not hard to see that \emph{all} measures we consider in this text thus satisfy (a). 
We recall the following lemma. 

\begin{lemma} \label{L: RP condition}
	Let $\c{H}_n:\mathbb{S}^{\mathbb{T}_n} \to \RR$ be the Hamiltonian of a spin-system on the torus satisfying
	\[
		-\c{H}_n = A + \Theta A + \sum_{i} C_i \Theta C_i
	\]
	for some functions $A, C_{i} \in \mathscr{U}^+$. Then $\mu_n \propto e^{-\c{H}_n}$ is reflection positive w.r.t. $\Theta$. 
\end{lemma}

For a proof we refer to e.g.~\cite{Biskup} or \cite[Lemma 10.8]{FriVel}. 
We point out already that for reflections going through vertices, 
the decomposition of Lemma \ref{L: RP condition} is trivial as we consider only nearest-neighbor interactions. 

For reflection through edges, we need that we can decompose
\[
	-\c{U}(t_x - t_y) = \sum_{i} F_{i}(t_x)F_{i}(t_y), 
\]
for some collection of functions $\{F_{i}\}$. 
By classical trigonometric identities, this can be easily deduced whenever $-\c{U}$ is positive definite and regular enough so that  
\[
	-\c{U}(t_x - t_y) = \sum_{i = 0}^\infty \alpha_i \cos(i(t_x - t_y)) = \sum_{i = 0}^\infty \alpha_i (\cos(it_x)\cos(it_y) + \sin(it_x)\sin(it_y)), 
\]
with $\alpha_i\geq0$.

\section{Positive definite functions} \label{ap:div}
We will call an even function $F:\mathbb{S} \to \RR$ \textit{conditionally positive definite} if for any vector $\xi = (\xi_1, \ldots, \xi_n) \in \RR^n$ with mean $0$ and all $t_1, \ldots, t_n \in \mathbb{S}$ it holds that
\[
	\sum_{i, j} \xi_i\xi_j F(t_i - t_j) \geq 0
\]
\begin{lemma}
	A function $F$ is conditionally positive definite if and only if $e^{cF}$ is positive definite for each $c > 0$. 
\end{lemma}
\begin{proof}
	Assume that $e^{cF}$ is p.d. for each $c > 0$. Then
	\[
		\frac{1}{c} \sum_{i, j} \xi_i\xi_j (e^{cF(t_i - t_j)} - 1) = \frac{1}{c} \sum_{i, j} \xi_i\xi_j e^{cF(t_i - t_j)}  \geq 0, 
	\]
	and taking $c \to 0$ shows one implication, since the derivative at zero of $e^{cF}$ is $F$. 
	The other implication follows from expanding the exponential and using that the space of conditional positive definite functions is closed under addition, 
	multiplication by nonnegative reals and multiplication. 
\end{proof}

Without proof, we will also state the following result. 

\begin{lemma} \label{L: cond PD -> PD}
	If $F:\mathbb{S} \to \RR$ is conditionally positive definite, then there exists a positive definite function $\varphi:\mathbb{S} \to \RR$ and a constant $c$ such that $F = \varphi - c$. 
\end{lemma}

\begin{proof}
	See e.g. Corollary 2.10.3 in \cite{BdlHV}. 
\end{proof}

Let us apply this to the potentials $\c{V}$ with dual potential $\c{U}$. Suppose that $\c{V}$ is divisible, i.e. that there exists a potential $\c{V}^{1}:\ZZ \to \RR$ such that
	\[
		e^{-\c{V}^1} * e^{-\c{V}^1} = e^{-\c{V}}. 
	\]
	This means that the Fourier transform $g$ of $e^{-\c{V}^1}$ satisfies $g^2 = e^{-\mathcal U}$. Assuming that $g$ is nonnegative, this implies in particular that the dual potential $\c{U}^{1}$ corresponding to $\c{V}^{1}$ equals $-\c{U}/2$. Moreover, $g = e^{-\c{U}/2}$ is positive definite (because its Fourier transform is non-negative). 
	
	Under the assumption that $e^{-\c{V}}$ is infinitely divisible, it is true that $\c{V}^1$ is also infinitely divisible. 
	Therefore, the $g$ is nowhere zero, which implies that its sign is fixed. However, $g(0) = 1$ so that $g > 0$ everywhere. 
	This implies in particular that $g = e^{-c \c{U}}$ must be positive definite for each $c > 0$ and hence we deduce that $-\c{U}$ is conditionally positive definite. 
	Lemma \ref{L: cond PD -> PD} then implies that we can take $-\c{U}$ to be positive definite since adding constants does not change the Gibbs measure.